\documentclass[11pt]{amsart}

\usepackage{geometry}
\usepackage{amsfonts, adjustbox, amssymb, amscd}
\numberwithin{equation}{section}

\usepackage{bm}
\usepackage{verbatim}
\usepackage{mathrsfs}
\usepackage{graphicx}
\usepackage{tikz-cd}
\usepackage{subcaption}
\usepackage{listings}
\usepackage{subfiles}
\usepackage[toc,page]{appendix}
\usepackage{mathtools}
\usepackage{comment}
\usepackage{enumerate}
\usepackage{enumitem}
\usepackage[all]{xy}

\usepackage{graphicx}
\usepackage{appendix}
\usepackage{hyperref}
\hypersetup{
    colorlinks=true,
    citecolor=red,
    linkcolor=blue,
    filecolor=magenta,      
    urlcolor=red,
}
\newcommand{\bir}{{\operatorname{bir}}}

\newcommand{\id}{\mathrm{id}}

\newcommand{\Rr}{\mathbb{R}}

\newcommand{\Center}{\operatorname{center}}

\newcommand{\Exc}{\operatorname{Exc}}

\newcommand{\Nklt}{\operatorname{Nklt}}

\newcommand{\lct}{\operatorname{lct}}

\newcommand{\Coreg}{\operatorname{Coreg}}
\newcommand{\SCoreg}{\operatorname{SCoreg}}
\newcommand{\bSCoreg}{\SCoreg^{\bir}}

\newcommand{\Supp}{\operatorname{Supp}}

\newcommand{\mult}{\operatorname{mult}}

\newcommand{\Bb}{{\bf{B}}}

\newcommand{\Ii}{\Gamma}

\newcommand{\reg}{\mathrm{reg}}
\newcommand{\Reg}{\mathrm{Reg}}

\newcommand{\SReg}{\mathrm{SReg}}

\newcommand{\bSReg}{\SReg^{\bir}}
\newcommand{\SRegt}{\mathrm{SRegt}}
\newcommand{\bSRegt}{\SRegt^{\bir}}


\newcounter{parentnumber}

\newtheorem{thm}{Theorem}[section]

\newtheorem{cor}[thm]{Corollary}
\newtheorem{lem}[thm]{Lemma}
\newtheorem{prop}[thm]{Proposition}

\newtheorem{claim}[thm]{Claim}

\theoremstyle{definition}
\newtheorem{defn}[thm]{Definition}
\newtheorem{ques}[thm]{Question}
\theoremstyle{definition}
\newtheorem{rem}[thm]{Remark}

\newtheorem{ex}[thm]{Example}

\newtheorem{nota}[thm]{Notation}

\theoremstyle{definition}

\begin{document}

\title{Dual complexes of qdlt Fano type models and strong complete regularity}

\author{Jihao Liu and Konstantin Loginov}

\date{\today}

\begin{abstract}
We introduce birational strong complete regularity and strong complete regularity, two numerical invariants for pairs of (relative) Fano type. They are defined using variants of qdlt Fano type models and the dimension of the dual complex of the reduced boundary, and can be viewed as Fano type refinements of Shokurov's complete regularity. We establish basic properties of these invariants and clarify its relation to models of qdlt Fano type appearing in K-stability. In particular, we prove that any pair with maximal birational strong complete regularity is $1$-complementary, and the thresholds where birational strong complete regularity or strong complete regularity jumps satisfy the ascending chain condition.
\end{abstract}

\address{Department of Mathematics, Peking University, No. 5 Yiheyuan Road, Haidian District, Peking 100871, China}
\email{liujihao@math.pku.edu.cn}

\address{{Steklov Mathematical Institute of Russian Academy of Sciences, Moscow, Russia.}
\newline
{Laboratory of AGHA, Moscow Institute of Physics and Technology.}}
\email{loginov@mi-ras.ru}

\maketitle

\pagestyle{myheadings}\markboth{\hfill Jihao Liu and Konstantin Loginov \hfill}{\hfill Dual complexes and strong complete regularity\hfill}

\tableofcontents

\section{Introduction}\label{sec: introduction}

Throughout the paper, we work over the field of complex numbers $\mathbb C$.

The notions of \emph{regularity} and \emph{complete regularity} were introduced by Shokurov in his study of the classification of threefold log canonical singularities and the boundedness of complements \cite[7.9 and 7.11]{Sho00}. In some references (cf. \cite{Mor24a}), by abuse of terminology, complete regularity is also called regularity; in this paper we distinguish the two and denote regularity by ``$\reg$'' and complete regularity by ``$\Reg$''. We refer the reader to Definition \ref{defn: regularity} for precise definitions. Roughly speaking, regularity (resp.\ complete regularity) describes the structure of a log Calabi--Yau (resp.\ $\mathbb R$-complementary) pair of the form $(X/Z,B)$ or $(X/Z\ni z,B)$. Here $\pi\colon X\to Z$ is a contraction, $z\in Z$ is a point, and $K_X+B\sim_{\mathbb R,Z}0$ (resp.\ $K_X+B^+\sim_{\mathbb R,Z}0$ for some $B^+\ge B$) over a neighborhood of $z$. In particular, when $Z=z$ is a closed point, regularity applies to (potentially) Calabi--Yau pairs $(X,B)$. When $\pi=\id_X$ and $z=x$ is a closed point of $X$, regularity (resp.\ complete regularity) applies to log canonical (resp.\ klt) singularities $(X\ni x,B)$.

Both regularity and complete regularity take integer values in the range $[-1,d-1]$, where $d=\dim X$. In \cite{FMP25}, Figueroa, Moraga, and Peng introduced (complete) \emph{coregularity} (denoted by $\Coreg$) for singularities (i.e.\ when $\pi=\id_X$) by setting 
$$\Coreg := d-1-\Reg,$$
so that $\Coreg\in\{0,\dots,d\}$. One motivation for considering coregularity is that pairs with the same coregularity in different dimensions often satisfy comparable bounds for certain numerical invariants. This notion was later extended to the general relative setting in \cite{Mor24a} (cf.\ \cite[paragraph before Remark 2.2]{Mor24a}). For related work on the classification of log canonical singularities and log Calabi--Yau pairs using regularity, complete regularity, and/or coregularity, see
\cite{Asc+23,FM23,GLM23,Mor23,ALP24,BF24,BL24,dS24,Mor24b,EFM24,HLS24,LMV24,LZ24,EJM25,FFMP25,FMP25,FMM25,LPT25,MM25,MYY25,Zha25}
and the references therein.

Despite their usefulness, complete regularity alone can be too coarse to distinguish singularities or pairs with genuinely different geometric behavior. We illustrate this phenomenon with the following examples.

\begin{ex}
Klt surface singularities are classified into types $A$, $D$, and $E$. Complete regularity distinguishes the $E$-type: the complete regularity of an $E$-type singularity equals $0$, whereas the complete regularity of both $A$-type and $D$-type equals $1$. However, this invariant does not distinguish $A$-type and $D$-type, even though their geometry is different: $A$-type singularities are toric, while $D$-type singularities are not.

One way to distinguish them is via the notion of \emph{complexity} \cite{BFMZ18}: any $A$-type singularity admits a complexity $0$ structure, while any lc pair with ambient space of $D$-type has complexity at least~$1$. See also \cite{MM25,MS25}. Another way is through complements: $A$-type singularities are $1$-complementary, whereas $D$-type singularities are not. We further discuss this in the next example.
\end{ex}

\begin{ex}
It is known that singularities with maximal complete regularity (equivalently, coregularity~$0$) are either $1$-complementary or $2$-complementary (cf.\ \cite{Asc+23,FFMP25}). However, not all such singularities are $1$-complementary (e.g.\ $D$-type surface singularities). This creates obstacles when one aims to characterize toric singularities: every toric singularity admits a $1$-complement with maximal regularity.

In practice, many results on characterizing toric varieties via complete regularity and dual complexes (cf.\ \cite{Mor24a,MM25}) therefore require passing to a $2$-to-$1$ cover, which effectively reduces $D$-type phenomena to $A$-type. This is, however, not ideal for characterizing toric structures without passing to coverings.
\end{ex}

\begin{ex}
In \cite{ALP24}, the second author together with Avilov and Przyjalkowski studied complete regularity for smooth del Pezzo surfaces (see also \cite{Mor24a}) and Fano threefolds. It is known that a smooth del Pezzo surface has coregularity $0$ if and only if its degree is $\ge 2$ (\cite[Theorem 3.1]{Mor24a}, \cite[Proposition 2.3]{ALP24}). For smooth Fano threefolds, in $100$ out of $105$ families, a general member has maximal complete regularity; moreover, among these $100$ families, in $92$ families every member has maximal complete regularity \cite[Theorem 0.1]{ALP24}. This indicates that complete regularity is often too weak to distinguish Fano varieties effectively, since most of them have maximal complete regularity.
\end{ex}

Motivated by these limitations, we introduce two refinements of complete regularity tailored to the geometry of Fano type varieties, which we call \emph{birational strong complete regularity} and \emph{strong complete regularity}.

\begin{defn}[Strong complete regularity]\label{defn: scr}
Let $(X,B)$ be a pair and $\pi\colon X\rightarrow Z$ a contraction. A \emph{birational qdlt Fano type model} (\emph{QF$^{\,\bir}$ model} for short) of $(X/Z,B)$ is a pair $(Y/Z,B_Y)$ together with a birational map $g\colon Y\dashrightarrow X$ over $Z$ satisfying the following.
\begin{enumerate}
    \item $g^{-1}$ does not contract any divisor and $B_Y\geq g^{-1}_*B$.
    \item $(Y,B_Y)$ is qdlt.
    \item $-(K_Y+B_Y)$ is ample$/Z$.
    \item Any $g$-exceptional prime divisor $D$ such that $\mult_DB_Y=0$ does not contain any lc center of $(Y,B_Y)$.
\end{enumerate}
In addition, if $g$ is a morphism, then we say that $(Y/Z,B_Y)$ is a \emph{qdlt Fano type model} (\emph{QF model} for short) of $(X/Z,B)$. We define the \emph{birational strong complete regularity} (resp.\ \emph{strong complete regularity}) of $(X/Z,B)$ by
$$\SReg^{\bir}(X/Z,B)
:=\max\Bigl\{-1, \dim \mathcal{D}(Y,B_Y)\Bigm|(Y/Z,B_Y)\text{ is a QF$^{\,\bir}$ model of }(X/Z,B)\Bigr\}$$
$$\left(\text{resp. }\SReg(X/Z,B)
:=\max\Bigl\{-1, \dim \mathcal{D}(Y,B_Y)\Bigm|(Y/Z,B_Y)\text{ is a QF model of }(X/Z,B)\Bigr\}\right)$$
where $\mathcal{D}(Y,B_Y)$ denotes the dual complex of $(Y,B_Y)$, cf.\ \cite[Definitions 8, 35]{dFKX17}.

We define the \emph{birational strong coregularity} (resp.\ \emph{strong coregularity}) of $(X/Z,B)$ as
$$\dim X-1-\bSReg(X/Z,B)\quad \left(\text{resp. }\dim X-1-\SReg(X/Z,B) \right).$$
We may similarly define (birational) strong complete regularity and strong coregularity for pairs of the form $(X/Z\ni z,B)$, where $Z\ni z$ is a germ. See Definition \ref{defn: regularity} for an explicit definition.
\end{defn}

Definition \ref{defn: scr} is motivated by K-stability. When $g$ is a birational morphism, the data $(Y,\lfloor B_Y\rfloor)$ satisfying (1-3) above agrees with the notion of a \emph{qdlt Fano type model} in \cite[Definition 3.5]{XZ25}. Our definition differs in two mild but useful ways: we allow birational maps (not only morphisms), and we impose the additional technical condition~(4).

Let us explain the role of~(4). Condition~(4) is designed to rule out \emph{inessential} exceptional divisors that are invisible to the boundary but nevertheless interact with the log canonical strata. Concretely, if $D$ is $g$-exceptional and $\mult_D B_Y=0$, then $D$ does not contribute to the dual complex $\mathcal D(Y,B_Y)$; however, if $D$ contains an lc center, it could affect how the lc strata behave under birational modifications and, in particular, obstruct reductions to more rigid models (see \cite[Proposition 3.6(2)]{XZ25} for a way of reduction). Condition~(4) ensures that all lc centers are accounted for by the qdlt boundary (and hence by the dual complex) in a way that is stable under the birational operations we use later. This becomes important when we want to connect qdlt Fano type models to stronger and more structured models in the spirit of \emph{Koll\'ar models} (cf.\ \cite[Definition 3.7]{XZ25}), which play a central role in the stable degeneration theorem as a generalization of plt blow-ups.

While Koll\'ar models in \cite{XZ25} are defined in the local singularity setting (typically $\pi=\id_X$ and $Z$ affine), our goals require working also in the global case ($Z$ a point) and, more generally, in relative settings. In these contexts, it is natural to allow birational maps $g$ (rather than only birational morphisms), and the additional condition~(4) provides the right framework to formulate and construct analogues of Koll\'ar models. Note that when $\pi=\id_X$, since $g$ is a map$/Z$, it is automatically a morphism. We refer the reader to Definition \ref{defn: reduced bqf}, \ref{defn: qag model}, and \ref{defn: qaa model} for the corresponding constructions (reduced QF models, QAA models, and QAG models), which will be used repeatedly in our arguments.

We also emphasize that our definitions are aimed at studying the structure of dual complexes attached to (birational) qdlt Fano type models. This perspective aligns with recent developments in local K-stability: lc places on qdlt Fano type models correspond to \emph{Koll\'ar valuations} in the sense of \cite[Definition 1.2]{LX24} and to \emph{special valuations} in the sense of \cite[32]{XZ26}; see also \cite{Che25}. It is expected that these lc places, and the associated dual complexes, exhibit richer intrinsic structures; see \cite[Question 1.5, Conjecture 1.8]{LX24} and \cite[Problem 35]{XZ26}. From this viewpoint, condition~(4) can be seen as a convenient requirement ensuring that the dual complex captures precisely the relevant lc geometry and behaves well under the reductions needed for our proofs.

Now let us compare (birational) strong complete regularity and complete regularity. Compared with complete regularity, (birational) strong complete regularity has both strengths and limitations due to the ampleness condition in Definition \ref{defn: scr}(3). In particular, the existence of a QF$^{(\bir)}$ model forces $X$ to be of Fano type (Lemma \ref{lem: QF is ft}). Consequently, strong complete regularity is less suited for studying Calabi--Yau varieties or strictly log canonical singularities (i.e.\ log canonical singularities that are not klt) in full generality. On the other hand, strong complete regularity is especially effective for Fano type varieties and klt singularities. For instance, it distinguishes $A$ and $D$-type surface singularities: the (birational) strong complete regularity of an $A$-type singularity is $1$, while the (birational) strong complete regularity of any $D$-type singularity is $0$.

In this paper, we study basic properties of (birational) strong complete regularity (see Section \ref{sec: strong complete regularity} for details). More importantly, we establish two results highlighting the intrinsic behavior of these invariants.

\medskip
\noindent\textbf{Maximal strong complete regularity and complements.}
As mentioned above, any pair with maximal complete regularity is either $1$-complementary or $2$-complementary \cite{Asc+23,FFMP25}. For maximal strong complete regularity we obtain a sharper statement.

\begin{thm}\label{thm: 1-complementary}
Let $(X/Z\ni z,B)$ be a pair such that
$$\bSReg(X/Z\ni z,B)\geq\dim X-1-\dim z.$$
Then $\bSReg(X/Z\ni z,B)=\dim X-1-\dim z$, and there exists a $1$-complement $(X/Z\ni z,B^+)$ of $(X/Z\ni z,B)$ such that
$$\reg(X/Z\ni z,B^+)=\dim X-1-\dim z.$$
In particular, $(X/Z\ni z,B)$ is $1$-complementary.
\end{thm}

It is worth mentioning that Theorem \ref{thm: 1-complementary} has no restriction on the coefficients of $B$. This differs from the setting of maximal complete regularity, where it is assumed that $B$ has standard coefficients (cf.\ \cite[Theorem 4]{FFMP25}).

\medskip
\noindent\textbf{ACC for strong complete regularity thresholds.}
Let $(X/Z,B)$ be an lc pair and let $D\geq 0$ be an $\mathbb R$-Cartier $\mathbb R$-divisor on $X$. Consider
$$\SReg^{(\bir)}(t):=\SReg^{(\bir)}(X/Z,B+tD), \qquad t\ge 0.$$
Then $\SReg^{(\bir)}(t)$ is non-increasing in $t$, and it is natural to study the thresholds where $\SReg^{(\bir)}(t)$ jumps. We formalize this via the following definition.

\begin{defn}\label{defn: scrt}
Let $d>n\geq 0$ be two integers. Let $(X/Z,B)$ be a pair of dimension $d$ and let $D\geq 0$ be an $\mathbb R$-Cartier $\mathbb R$-divisor on $X$. The $n$-th (birational) strong complete regularity threshold of $D$ with respect to $(X/Z,B)$ is
$$\SRegt^{(\bir)}_n(X/Z,B;D):=\inf\{\,t\ge 0 \mid \SReg^{(\bir)}(X/Z,B+tD)<n\},$$
with the convention that the infimum of the empty set is $+\infty$.
\end{defn}

We show that these thresholds satisfy the ACC.

\begin{thm}\label{thm: acc scrt}
Let $d$ be a positive integer and $\Ii\subset [0,+\infty)$ a DCC set. Then there exists an ACC set $\Ii'$ depending only on $d$ and $\Ii$ satisfying the following.

Let $(X/Z,B)$ be a pair of dimension $d$ and $D\geq 0$ an $\mathbb R$-Cartier $\mathbb R$-divisor on $X$. Assume that the coefficients of $B$ and $D$ belong to $\Ii$. Then for any integer $n$ such that $0\leq n<d$, we have
$$\bSRegt_n(X/Z,B;D)\in\Ii'\quad \text{and}\quad \SRegt_n(X/Z,B;D)\in\Ii'.$$
\end{thm}

We remark that an analogous statement, the ACC for complete regularity thresholds on (relative) Fano type varieties, is proved in \cite[Theorem 8.25]{HLS19}. This should be distinguished from the ``ACC for lc thresholds for fixed coregularity'' in \cite[Theorem 1]{FMP25}. Theorem \ref{thm: acc scrt} may also be viewed as an analogue of the ACC for lc thresholds \cite[Theorem 1.1]{HMX14} and the ACC for $\mathbb R$-complementary thresholds \cite[Theorem 5.12]{HLS24}, \cite[Theorem 21]{Sho20}.

\subsection*{Structure of the paper}
The paper is organized as follows. In Section \ref{sec: preliminaries}, we recall preliminaries that will be used in the rest of the paper. In Section \ref{sec: strong complete regularity}, we establish basic properties of strong complete regularity. In Section \ref{sec: complement}, we prove that pairs with maximal strong complete regularity are $1$-complementary (Theorem \ref{thm: 1-complementary}). In Section \ref{sec: acc}, we prove the ACC for strong complete regularity thresholds (Theorem \ref{thm: acc scrt}).

\subsection*{Acknowledgments}
This work is supported by the National Key R\&D Program of China \#2024YFA1014400. The work of the second author was performed at the Steklov International Mathematical Center and supported by the Ministry of Science and Higher Education of the Russian Federation (agreement no. 075-15-2022-265), and the Simons Foundation. The authors would like to thank Eduardo Alves da Silva for useful discussions.

\section{Preliminaries}\label{sec: preliminaries}
This section fixes notation and collects a few standard facts that will be used throughout the paper. We adopt the standard notation and terminology for the minimal model program from \cite{Sho92,KM98,BCHM10} and use them freely.

\subsection{Basic notation}

\begin{defn}
A \emph{contraction} $f\colon X\rightarrow Z$ is a projective morphism such that $f_*\mathcal{O}_X=\mathcal{O}_Z$. 
\end{defn}

\begin{defn}
    Let $X\rightarrow Z$ be a contraction between normal quasi-projective varieties and let $D$ be an $\mathbb R$-divisor on $X$. The \emph{stable base locus}$/Z$ of $D$ is defined as
    $$\Bb(D/Z):=X\bigcap\left(\bigcap_{0\leq D'\sim_{\mathbb R,Z}D}\Supp D'\right).$$
    The \emph{augmented base locus}$/Z$ of $D$ is defined as
    $$\Bb_+(D/Z):=\Bb((D-\epsilon A)/Z)$$
    for any ample$/Z$ divisor $A$ on $X$ and any $0<\epsilon\ll 1$. It is well-known that $\Bb_+(D/Z)$ is a Zariski closed subset and is independent of the choice of $A$ \cite[Definition 3.5.1]{BCHM10}.
\end{defn}

\begin{nota}
    Let $X$ be a normal quasi-projective variety and $D$ an $\mathbb R$-divisor on $X$. Write $D=\sum d_iD_i$ where $D_i$ are distinct prime divisors. For any real number $a$, we denote by $D^{=a}$ (resp. $D^{<a}$, $D^{\geq a}$) as
    $$\sum_{i\colon d_i=a}d_iD_i\left(\text{resp. }\sum_{i\colon d_i<a}d_iD_i,\sum_{i\colon d_i\geq a}d_iD_i\right).$$
    For any $D'=\sum d_i'D_i$, we denote by
    $$D\wedge D':=\sum\min\{d_i,d_i'\}D_i\quad \text{and}\quad D\vee D':=\sum\max\{d_i,d_i'\}D_i.$$
\end{nota}

\begin{nota}
Let $\phi\colon X\dashrightarrow X'$ be a birational map between normal quasi-projective varieties. We denote by $\Exc(h)$ the reduced divisor supported on the codimension one part of the exceptional locus of $\phi$. 
\end{nota}

\begin{nota}
    We denote by $\dim\emptyset=-1$.
\end{nota}

\begin{nota}\label{nota: closure of set}
    For any set $\Ii\subset\mathbb R$, we denote by $\overline{\Ii}$ the closure of $\Ii$ in $\mathbb R$.
\end{nota}

\subsection{Pairs}

\begin{defn}
A \emph{sub-pair} (reps. \emph{pair}) $(X/Z\ni z,B)$ consists of a contraction $\pi\colon X\rightarrow Z$ between normal quasi-projective varieties, a (not necessarily closed) point $z\in Z$, and an $\mathbb R$-divisor $B$ (resp. $B\geq 0$) such that $K_X+B$ is $\mathbb R$-Cartier. If $Z\ni z$ is not important, then we may omit it and say that $(X,B)$ is a sub-pair (resp. pair). If $(X/Z\ni z,B)$ is a sub-pair (resp. pair) for any point $z\in Z$, then we say that $(X/Z,B)$ is sub-pair (resp. pair).

Let $(X,B)$ be a sub-pair. For any projective birational morphism $h\colon X'\rightarrow X$ with
$$K_{X'}+B'=h^*(K_X+B)$$
and any prime divisor $D$ on $X'$, we denote by $a(D,X,B):=1-\mult_DB'$ the \emph{log discrepancy} of $D$ with respect to $(X,B)$. An \emph{nklt place} (resp. \emph{lc place}) of $(X,B)$ is a prime divisor over $X$ such that $a(D,X,B)\leq 0$ (resp. $=0$). An \emph{nklt center} of $(X,B)$ is the center of an nklt place of $(X,B)$ on $X$. The union of all nklt centers of $(X,B)$ is called the \emph{nklt locus} of $(X,B)$ and is denoted by $\Nklt(X,B)$. 

We say that $(X,B)$ is lc (resp. klt) if $a(D,X,B)\geq 0$ (resp. $>0$) for any prime divisor $D$ over $X$. We say that $(X,B)$ is \emph{plt} if $a(D,X,B)>0$ for any prime divisor $D$ that is exceptional$/X$. We say that $(X,B)$ is \emph{qdlt} if $(X,B)$ is lc, and there exists an open subset $U\subset X$ such that
\begin{itemize}
    \item $U$ contains the generic point of any lc center of $(X,B)$, and
    \item $(U,B|_U)$ is $\mathbb Q$-factorial log toroidal
\end{itemize}
(cf. \cite[Definition 35]{dFKX17}).
\end{defn}

\begin{defn}
    Let $(X,B)$ and $(X',B')$ be two sub-pairs and $\phi\colon X\dashrightarrow X'$ a birational map. We say that  $(X,B)$ and $(X',B')$ are \emph{crepant} if for any resolution of indeterminacy $p\colon  W\rightarrow X$ and $q\colon W\rightarrow X'$ of $\phi$, we have
    $$p^*(K_X+B)=q^*(K_{X'}+B').$$
\end{defn}

\begin{defn}
Let $\pi\colon X\rightarrow Z$ be a contraction between normal quasi-projective varieties and $z\in Z$ a (not necessarily closed) point. 

A pair $(X/Z,B)$ is called \emph{Fano} (resp. \emph{Calabi-Yau}) if $-(K_X+B)$ is ample$/Z$ (resp. $K_X+B\sim_{\mathbb R,Z}0$). We say that $X/Z$ is \emph{of Fano type} if there exists a klt Fano pair $(X/Z,\Delta)$, and we also say that $X$ is \emph{of Fano type}$/Z$. 

We say that $X$ is \emph{potentially klt} if there exists a klt pair $(X,\Delta)$.
\end{defn}
%A pair $(X/Z,B)$ is called \emph{$\mathbb R$-complementary} if there exists an lc Calabi-Yau pair $(X/Z,B^+)$ such that $B^+\geq B$

\begin{defn}
Let $(X/Z,B)$ be a pair and $n$ a positive integer. An \emph{$\mathbb R$-complement} of $(X/Z,B)$ is an lc Calabi-Yau pair $(X/Z,B^+)$ such $B^+\geq B$. A \emph{$\mathbb Q$-complement} of $(X/Z,B)$ is an $\mathbb R$-complement $(X/Z,B^+)$ of $(X/Z,B)$ such that $B^+$ is a $\mathbb Q$-divisor; note that in this case we automatically get $K_X+B^+\sim_{\mathbb Q,Z}0$, cf. \cite[Lemma 5.3]{HLS24}. An \emph{$n$-complement} of $(X/Z,B)$ is an lc Calabi-Yau pair $(X/Z,B_n)$ such that $n(K_X+B_n)\sim_Z 0$ and 
$$nB_n\geq\left\lfloor (n+1)B^{<1}\right\rfloor+n B^{=1}.$$
A \emph{monotonic $n$-complement} of $(X/Z,B)$ is an $n$-complement of $(X/Z,B)$ that is also a $\mathbb Q$-complement of $(X/Z,B)$. If $(X/Z,B)$ has an $\mathbb R$-complement (resp. $n$-complement), then we say that $(X/Z,B)$ is \emph{$\mathbb R$-complementary} (resp. \emph{$n$-complementary}).

Let $(X/Z\ni z,B)$ be a pair with contraction $\pi\colon X\rightarrow Z$. An \emph{$\mathbb R$-complement} (resp. \emph{$n$-complement}, \emph{monotonic $n$-complement}) of $(X/Z\ni z,B)$ is a pair $(X/Z\ni z,B^+)$, such that
$\left(\pi^{-1}(U)/U,B^+|_{\pi^{-1}(U)}\right)$
is an $\mathbb R$-complement (resp. $n$-complement, monotonic $n$-complement) of $\left(\pi^{-1}(U)/U,B|_{\pi^{-1}(U)}\right)$ for some open neighborhood $U$ of $z$ in $Z$. We say that $(X/Z\ni z,B)$ is \emph{$\mathbb R$-complementary} (resp. \emph{$n$-complementary}) if $(X/Z\ni z,B)$ has an $\mathbb R$-complement (resp. $n$-complement).
\end{defn}

\subsection{Relative Nakayama-Zariski decomposition}

\begin{defn}
    Let $\pi\colon X\rightarrow U$ be a projective morphism between normal quasi-projective varieties, $D$ a pseudo-effective$/U$ $\Rr$-Cartier $\Rr$-divisor on $X$, and $P$ a prime divisor on $X$. We define $\sigma_{P}(X/U,D)$ as in \cite[Definition 3.1]{LX25} by considering $\sigma_{P}(X/U,D)$ as a number in  $[0,+\infty)\cup\{+\infty\}$. We define $$N_{\sigma}(X/U,D)=\sum_Q\sigma_Q(X/U,D)Q$$
    where the sum runs through all prime divisors on $X$ and consider it as a formal sum of divisors with coefficients in $[0,+\infty)\cup\{+\infty\}$. 
\end{defn}

\section{Strong complete regularity}\label{sec: strong complete regularity}

The purpose of this section is to develop the basic tools needed to work with (birational) strong complete regularity.
Although these invariants are defined by taking a maximum over (birational) qdlt Fano type models, their effective use requires several technical refinements of these models.
In particular, we will introduce \emph{reduced} (birational) qdlt Fano type models, which allow us to control the reduced boundary and its dual complex, and we will also introduce \emph{qdlt anti-ample models}, which are better suited for running MMP arguments and can be considered as a generalization of Koll\'ar models in the sense of \cite[Definition 3.7]{XZ25}. These refinements will be used repeatedly in the proofs of the main theorems in Sections \ref{sec: complement} and \ref{sec: acc}.

\subsection{Definitions and variants}

In this subsection we fix terminology for (birational) qdlt Fano type models and recall the definitions of regularity, complete regularity, and their strong variants.

\begin{defn}\label{defn: qdlt model}
Let $(X/Z,B)$ be a pair. A \emph{QF$^{\,\bir}$ model} of $(X/Z,B)$ is of the form $g\colon (Y,B_Y)\dashrightarrow (X,B)$ so that $g$ is a birational map and the following conditions hold.
\begin{enumerate}
    \item $g$ does not extract any divisor and $B_Y\geq g^{-1}_*B$.
    \item $(Y/Z,B_Y)$ is a qdlt pair.
    \item $-(K_Y+B_Y)$ is ample$/Z$.
    \item Any $g$-exceptional prime divisor $D$ such that $\mult_DB_Y=0$ does not pass through any lc center of $(Y,B_Y)$.
\end{enumerate}
We also say that $(Y/Z,B_Y)$ is a \emph{QF$^{\,\bir}$ model} of $(X/Z,B)$ associated with $g$. In addition, if $g$ is a morphism, then we say that $g\colon (Y,B_Y)\rightarrow (X,B)$ is a \emph{QF model} of $(X/Z,B)$, and also say that $(Y/Z,B_Y)$ is a \emph{QF model} of $(X/Z,B)$ associated with $g$.
\end{defn}

\begin{rem}
    In the Definition \ref{defn: qdlt model} ``QF$^{(\bir)}$" stands for ``(Birational) Qdlt Fano". We use this simplified notation because the phrase ``model of qdlt Fano type" is too long. 
    
    Be aware that the definition of QF models in Definition \ref{defn: qdlt model} has a small difference with \cite[Definition 3.5]{XZ25}, where they use $(Y/Z,B_Y^{=1})$ instead of  $(Y/Z,B_Y)$. Our definition also has a small difference with \cite[32]{XZ26} due to condition (4). On the other hand, in practice, it is important to ensure that the birational maps we may consider are isomorphic near the generic point of the lc centers of $(Y,B_Y)$ (cf. \cite[Proposition 3.6(2)]{XZ25}), while the most common QF$^{\,\bir}$ models we have are those satisfying that $B_Y=g^{-1}_*B+\Exc(g)$, i.e. reduced QF$^{\,\bir}$ models which we will introduced later (Definition \ref{defn: reduced bqf}). Therefore, it is natural to add condition (4).
\end{rem}

\begin{defn}\label{defn: regularity}
We refer the reader to \cite[Definition 8]{dFKX17} for the definition of the \emph{dual complex} $\mathcal{D}(\Delta)$ for any snc reduced divisor $\Delta$ on a smooth variety. 

Let $(X/Z,B)$ be an lc sub-pair. Let $f\colon W\rightarrow X$ be a log resolution of $(X,B)$ and write
$$K_W+B_W=f^*(K_X+B)$$
over a neighborhood of $z$. We denote by
$$\reg(X,B):=\dim\mathcal{D}\left(B_W^{=1}\right)$$
the \emph{regularity} of $(X,B)$. We also denote by $\reg(X/Z,B):=\reg(X,B)$ and say that $\reg(X/Z,B)$ is the regularity of $(X/Z,B)$. This value is well-defined and is independent of the choice of $f$ as the PL-homeomorphism class of $\mathcal{D}(B_W^{=1})$ does not depend on the choice of $f$ (cf. \cite[Theorems 1, 28]{dFKX17}). 

Let $(X/Z\ni z,B)$ be a sub-pair with a contraction $\pi\colon X\rightarrow Z$. If $(X/Z\ni z,B)$ is lc, then we denote by $\reg(X/Z\ni z,B)$, its regularity as the infimum of $\reg(\pi^{-1}(U)/U,B|_{\pi^{-1}(U)})$ over all open neighborhoods $U$ of $z$ such that $(\pi^{-1}(U),B|_{\pi^{-1}(U)})$ is lc. 

Let $(X/Z,B)$ be a pair. We denote by
$$\Reg(X/Z,B):=\max\{-1,\reg(X,B^+)\mid (X/Z,B^+)\text{ is an }\mathbb R\text{-complement of }(X/Z,B)\}$$
the \emph{complete regularity} of $(X/Z,B)$. We denote by
$$\bSReg(X/Z,B):=\max\{-1,\reg(Y,B_Y)\mid (Y/Z,B_Y)\text{ is a QF$^{\,\bir}$ model of }(X/Z,B)\}$$
and
$$\SReg(X/Z,B):=\max\{-1,\reg(Y,B_Y)\mid (Y/Z,B_Y)\text{ is a QF model of }(X/Z,B)\}$$
the \emph{birational strong complete regularity} and \emph{strong complete regularity} of $(X/Z,B)$, respectively. Note that, since $(Y,B_Y)$ is qdlt, $\reg(Y,B_Y)$ is the dimension of the dual complex $\mathcal{D}(B_Y^{=1})$ defined as in \cite[Paragraph before Lemma 36]{dFKX17}. We define  
$$\bSCoreg(X/Z,B):=\dim X-1-\bSReg(X/Z,B)$$ 
and
$$\SCoreg(X/Z,B):=\dim X-1-\SReg(X/Z,B)$$
as the \emph{birational strong coregularity} and \emph{strong coregularity} of $(X/Z,B)$ respectively.

Let $(X/Z\ni z,B)$ be a pair. We denote by 
\begin{align*}
&\Reg(X/Z\ni z,B)\\
:=&\max\{-1,\reg(X/Z\ni z,B^+)\mid (X/Z\ni z,B^+)\text{ is an }\mathbb R\text{-complement of }(X/Z\ni z,B)\}
\end{align*}
the \emph{complete regularity} of $(X/Z\ni z,B)$. We denote by 
\begin{align*}
   &\bSReg(X/Z\ni z,B)\\
   :=&\max\{-1,\reg(V/U\ni z,B_V)\mid (V/U,B_V)\text{ is a QF$^{\,\bir}$ model of }(\pi^{-1}(U)/U,B|_{\pi^{-1}(U)})\} 
\end{align*}
and
\begin{align*}
   &\SReg(X/Z\ni z,B)\\
   :=&\max\{-1,\reg(V/U\ni z,B_V)\mid (V/U,B_V)\text{ is a QF model of }(\pi^{-1}(U)/U,B|_{\pi^{-1}(U)})\} 
\end{align*}
the \emph{birational strong complete coregularity} and \emph{strong complete coregularity} of $(X/Z\ni z,B)$ respectively, where $U$ runs through all open neighborhoods of $z$ in $Z$. We define 
$$\bSCoreg(X/Z\ni z,B):=\dim X-1-\dim z-\bSReg(X/Z\ni z,B)$$
and
$$\SCoreg(X/Z\ni z,B):=\dim X-1-\dim z-\SReg(X/Z\ni z,B)$$
as the \emph{birational strong complete coregularity} and \emph{strong complete coregularity} of $(X/Z\ni z,B)$ respectively.
\end{defn}

We remark that in some references (cf. \cite{Mor24a}) complete regularity and complete coregularity are defined only for pairs $(X/Z,B)$ or $(X/Z\ni z,B)$ when $B$ is a $\mathbb Q$-divisor and only consider $\mathbb Q$-complements instead of $\mathbb R$-complements. The following lemma shows that Definition \ref{defn: regularity} is consistent with the definition of complete regularity and coregularity in other references.

\begin{lem}
Let $(X/Z,B)$ be a pair such that $B$ is a $\mathbb Q$-divisor. Then
$$\Reg(X/Z,B)=c_{\mathbb Q}:=\max\{-1,\reg(X,B^+)\mid (X/Z,B^+)\text{ is a }\mathbb Q\text{-complement of }(X/Z,B)\}.$$
\end{lem}
\begin{proof}
Since any $\mathbb Q$-complement of $(X/Z,B)$ is an $\mathbb R$-complement of $(X/Z,B)$, we have $c_{\mathbb Q}\leq\Reg(X/Z,B)$. In the following, we show that $c_{\mathbb Q}\geq\Reg(X/Z,B)$ which implies the lemma. We may assume that $\Reg(X/Z,B)\geq 0$. Then there exists an $\mathbb R$-complement $(X/Z,B^+)$ of $(X/Z,B)$ such that $\reg(X,B^+)=\Reg(X/Z,B)$. We write $B^+=\sum_{i=1}^m b_iB_i$ where $B_i$ are the irreducible components of $B^+$. Let $\bm{v}_0:=(b_1,\dots,b_m)$ and let $V\ni\bm{v}_0$ be the rational envelope of $\bm{v}_0$ in $\mathbb R^m$. Let $B(\bm{v}):=\sum_{i=1}^mv_iB_i$ for any $\bm{v}:=(v_1,\dots,v_m)$. 
    
    By \cite[Lemmas 5.3 and 5.4]{HLS24} and \cite[Theorem 5.6]{HLS24}, there exists an open subset $U\ni\bm{v}_0$ of $V$, such that  for any $\bm{v}\in U$, we have that $(X/Z,B(\bm{v}))$ is lc, $K_X+B(\bm{v})\sim_{\mathbb R,Z}0$, and any lc place of $(X,B^+)$ is an lc place of $(X,B(\bm{v}))$. In particular,
    $$\reg(X,B^+)\leq\reg(X,B(\bm{v}))$$
    for any $\bm{v}\in U$. Since $B$ is a $\mathbb Q$-divisor, possibly shrinking $U$, we may assume that $B(\bm{v})\geq B$ for any $\bm{v}\in U$. Take $\bm{v}_1\in U\cap\mathbb Q^m$, then $(X/Z,B(\bm{v}_1))$ is a $\mathbb Q$-complement of $(X/Z,B)$. We have
    $$\Reg(X/Z,B)=\reg(X,B^+)\leq\reg(X,B(\bm{v}_1))\leq c_{\mathbb Q}.$$
    The lemma follows.
\end{proof}

\subsection{Comparing complete regularities}

This subsection studies the basic properties of (birational) strong complete regularity, and compares (birational) strong complete regularity with Shokurov's complete regularity in the settings where both are defined. We start with the following lemma which implies that the birational strong complete regularity and strong complete regularity can always be obtained by $\mathbb R$-complements.
\begin{lem}\label{lem: bscr obtained}
Let $(X/Z,B)$ be a pair and $g\colon (Y,B_Y)\dashrightarrow (X,B)$ a QF$^{\,\bir}$ model of $(X/Z,B)$. Then there exists an $\mathbb R$-complement $(X/Z,B^+)$ of $(X/Z,B)$ such that
$$\reg(X,B^+)=\reg(Y,B_Y).$$
\end{lem}
\begin{proof}
Let $p\colon W\rightarrow Y$ and $q\colon W\rightarrow X$ be a resolution of indeterminacy of $g$. Since $-(K_Y+B_Y)$ is ample$/Z$, there exists $0\leq H_Y\sim_{\mathbb R,Z}-(K_Y+B_Y)$ such that $(Y,\Delta_Y:=B_Y+H_Y)$ is lc and $\mathcal{D}(Y,\Delta_Y)=\mathcal{D}(Y,B_Y)$. We let $\Delta:=g_*\Delta_Y$ and write
$$0\sim_{\mathbb R,Z}p^*(K_Y+\Delta_Y)=\colon q^*(K_X+\Delta)+E.$$
Then $E$ is exceptional$/X$. By applying the negativity lemma twice, we have that $E=0$, so $(Y,\Delta_Y)$ and $(X,\Delta)$ are crepant. Since $\Delta\geq g_*B_Y\geq B$, we have
$$\reg(Y,B_Y)=\reg(Y,\Delta_Y)=\reg(X,\Delta)$$
The lemma follows by taking $B^+:=\Delta$.
\end{proof}

\begin{prop}\label{prop: compare creg}
    Let $(X/Z,B)$ (resp. $(X/Z\ni z,B)$) be a pair. Then
    $$\dim X-1\geq\Reg(X/Z,B)\geq\bSReg(X/Z,B)\geq\SReg(X/Z,B).$$
    $$\left(\text{resp. }\dim X-1-\dim z\geq\Reg(X/Z,B)\geq\bSReg(X/Z,B)\geq\SReg(X/Z,B).\right)$$  
\end{prop}
\begin{proof}
Since the QF model for any pair is also a QF$^{\,\bir}$ model, we have $$\bSReg(X/Z,B)\geq\SReg(X/Z,B)\quad \left(\text{resp. }\bSReg(X/Z\ni z,B)\geq\SReg(X/Z\ni z,B)\right).$$
By Lemma \ref{lem: bscr obtained}, we have $$\Reg(X/Z,B)\geq\bSReg(X/Z,B)\quad \left(\text{resp. }\Reg(X/Z\ni z,B)\geq\bSReg(X/Z\ni z,B)\right).$$
For any snc divisor $D$ on a smooth variety $Y$, any stratum of $D$ is the intersection of $\leq \dim Y-1$ irreducible components of $D$, and for any contraction $f\colon Y\rightarrow Z$ and point $z\in Z$, any stratum $V$ of $D$ such that $z\in f(V)$ is the intersection of $\leq \dim Y-1-\dim z$ irreducible components of $D$. Therefore,
$\Reg(X/Z,B)\leq\dim X-1$ (resp. $\Reg(X/Z\ni z,B)\leq\dim X-1-\dim z$)
by the definition of complete regularity. The proposition follows.
\end{proof}

\begin{cor}\label{cor: bscr nonnegativ imply lc}
   Let $(X/Z,B)$ be a pair such that $\bSReg(X/Z,B)\geq 0$. Then $(X,B)$ is lc. 
\end{cor}
\begin{proof}
By Proposition \ref{prop: compare creg}, $\Reg(X/Z,B)\geq 0$, so $(X/Z,B)$ is $\mathbb R$-complementary and hence lc.
\end{proof}

It remains interesting to ask when the inequalities in Proposition \ref{prop: compare creg} are equalities. We ask the following question:

\begin{ques}\label{ques: bscr=scr}
    Let $(X/Z,B)$ be a pair. Do we always have
    $\bSReg(X/Z,B)=\SReg(X/Z,B)$?
\end{ques}

The following example shows some subtlety of Question \ref{ques: bscr=scr}. Roughly speaking, the key difficulty for Question \ref{ques: bscr=scr} arises when anti-flips or flops appear. This is essentially because flops do not preserve qdlt property (cf. \cite[Example 3.13.9]{Fuj17}, \cite[Section 4]{Has20}).

\begin{ex}
    Let $(Y,B_Y)$ be a $\mathbb Q$-factorial dlt pair and $f\colon Y\rightarrow Z$ a $(K_Y+B_Y)$-flipping contraction such that an lc center $V$ of $(Y,B_Y)$ is contained in the flipping locus of $f$. Let $f^+\colon X\rightarrow Z$ be the flip of $f$ associated with birational map $g\colon  Y\dashrightarrow X$ and let $B:=\phi_*B_Y$.

    Consider the pair $(X/Z,B)$. We have that $g\colon (Y,B_Y)\dashrightarrow (X,B)$ is a QF$^{\,\bir}$ model of $(X/Z,B)$, so $\bSReg(X/Z,B)\geq 0$. Now suppose that there exists a QF model $h\colon (W,B_W)\rightarrow (X,B)$ of $(X/Z,B)$. We let $B':=h_*B_W$. Since $-(K_W+B_W)$ is ample$/Z$, by the negativity lemma we have
    $$K_W+B_W=h^*(K_X+B')+F$$
    for some $F\geq 0$ such that $\Exc(h)=\Supp F$.

    We have that $B'\geq B$. We claim that $K_X+B'$ is nef$/Z$. To see this, we let $B_Y'$ be the strict transform of $B'$ on $Y$. Then $B_Y'\geq B_Y$. If $K_X+B'$ is not nef$/Z$, then $K_Y+B_Y'$ is not ample$/Z$, so $g$ is $(K_Y+B_Y')$-positive. We let $D$ be an lc place of $(Y,B_Y)$ such that $\Center_YD=V$, then we have
    $$0=a(D,Y,B_Y)\geq a(D,Y,B_Y')>a(D,X,B')\geq a(D,W,B_W)\geq 0,$$
    which is not possible. This indicates that $-F$ is ample$/Z$. In particular, $h$ cannot be the identity morphism.
\end{ex}

One particularly interesting case is when the associated contraction $X\rightarrow Z$ is the identity morphism. This is what all related studies on QF models up to now \cite{LX24,Che25,XZ25,XZ26} mainly focus on, since the corresponding K-stability theory only focuses on the study of the dual complex of potential lc places of klt singularities. In this case, the birational strong complete regularity equals to the strong complete regularity.

\begin{prop}\label{prop: bscr=scr when identity}
Let $(X,B)$ be a pair. Then $\bSReg(X/X,B)=\SReg(X/X,B)$.
\end{prop}
\begin{proof}
It follows from the fact that any birational QF model of $(X/X,B)$ is also a QF model of $(X/X,B)$.
\end{proof}

Similarly, Question \ref{ques: bscr=scr} has a positive answer when $\dim X=2$.

\begin{prop}\label{prop: surface bscr=scr}
   Let $(X/Z,B)$ be a pair such that $\dim X=2$. Then any QF$^{\,\bir}$ model of $(X/Z,B)$ is a QF model of $(X/Z,B)$. In particular, $\bSReg(X/Z,B)=\SReg(X/Z,B)$. 
\end{prop}
\begin{proof}
    Let $g\colon (Y,B_Y)\dashrightarrow (X,B)$ be a QF$^{\,\bir}$ model of $(X/Z,B)$. Since $g$ does not extract any divisor and $\dim X=2$, $g$ is a morphism. Thus $g\colon (Y,B_Y)\rightarrow (X,B)$ is a QF model of $(X/Z,B)$. The proposition follows.
\end{proof}

On the other hand, in general, complete regularity and (birational) strong complete regularity are usually different. This can be seen from the following example.

\begin{ex}
Let $X\ni x$ be a $D$-type surface singularity. Then $\Reg(X\ni x,0)=1$ because it is not exceptional. However, $\SReg(X\ni x,0)=0$. To see this, note that $D$-type singularity is weakly exceptional and only has one Koll\'ar component, hence it only has one reduced QF model (see Definition \ref{defn: reduced bqf} below) whose regularity is $0$.  By Proposition \ref{prop: surface bscr=scr}, $\SReg(X\ni x,0)=\bSReg(X\ni x,0)=0$.
\end{ex}

We conclude this subsection by considering the following proposition. It is an easy observation but indicates that birational strong complete regularity is a birational invariant that is preserved under the MMP.

\begin{prop}\label{prop: compare scr under contraction}
    Let $(X/Z,B)$ be a pair and $\phi\colon X\dashrightarrow X'$ a birational map$/Z$ (resp. birational morphism$/Z$) which does not extract any divisor and $K_{X'}+\phi_*B$ is $\mathbb R$-Cartier. Then any QF$^{\,\bir}$ model (resp. QF model) of $(X/Z,B)$ is a QF$^{\,\bir}$ model (resp. QF model) of $(X'/Z,\phi_*B)$. In particular,
    $$\bSReg(X/Z,B)\leq\bSReg(X'/Z,\phi_*B)\quad \text{(resp. }\SReg(X/Z,B)\leq\SReg(X'/Z,\phi_*B)\text{)}.$$
\end{prop}
\begin{proof}
    It immediately follows from the definition of QF$^{\,\bir}$ models and QF models.
\end{proof}

As an immediate consequence, we know that birational strong complete regularity is an invariant for small modifications. In particular, birational strong complete regularity is preserved under flips and small $\mathbb Q$-factorializations.
\begin{cor}\label{cor: compare scr under contraction small}
 Let $(X/Z,B)$ be a pair and $\phi\colon X\dashrightarrow X'$ a small birational  map$/Z$ such that $K_{X'}+\phi_*B$ is $\mathbb R$-Cartier. Then
 $$\bSReg(X/Z,B)=\bSReg(X'/Z,\phi_*B).$$
\end{cor}
\begin{proof}
It follows from Proposition \ref{prop: compare scr under contraction} by applying to $\phi$ and $\phi^{-1}$.
\end{proof}

\subsection{Reduced (birational) qdlt Fano type models}

In Definition \ref{defn: qdlt model}, the boundary $B_Y$ of a QF$^{(\bir)}$ model is only required to satisfy $B_Y\ge g^{-1}_*B$, so there is no canonical choice of $B_Y$.
For several arguments later (especially those involving boundedness results for dual complexes), it is convenient to canonicalize $B_Y$. The goal of this subsection is to introduce \emph{reduced} QF$^{(\bir)}$ models and to show that (birational) strong complete regularity can be computed using them. We begin with the following definition:

\begin{defn}\label{defn: reduced bqf}
Let $(X/Z,B)$ be a pair and $g\colon (Y,B_Y)\dashrightarrow (X,B)$ a QF$^{(\bir)}$ model of $(X/Z,B)$. We say that $(Y/Z,B_Y)$ is a \emph{reduced} QF$^{(\bir)}$ model of $(X/Z,B)$ if 
$$B_Y\geq g^{-1}_*B+\Exc(g).$$
\end{defn}

The goal of this subsection is to prove the following theorem:

\begin{thm}\label{thm: bqf induces reduced BQF}
Let $(X/Z,B)$ be a pair and $g\colon (Y,B_Y)\dashrightarrow (X,B)$ a QF$^{(\bir)}$ model of $(X/Z,B)$. Then there exists a reduced QF$^{(\bir)}$ model $h\colon (X',B')\dashrightarrow (X,B)$ of $(X/Z,B)$ satisfying the following.
\begin{enumerate}
\item Any lc place of $(Y,B_Y)$ is an lc place of $(X',B')$ and any lc place of $(X',B')$ is an lc place of $(Y,B_Y)$, and $\mathcal{D}(Y,B_Y)$ and $\mathcal{D}(X',B')$ are PL-homeomorphic.
\item $X'$ is $\mathbb Q$-factorial.
\item The induced birational map $\phi\colon Y\dashrightarrow X'$ does not extract any divisor and $$\Exc(\phi)=\Exc(g)-B_Y^{=1}\wedge\Exc(g).$$
\end{enumerate}
\end{thm}

Before we prove Theorem \ref{thm: bqf induces reduced BQF}, we prove the following useful lemmas. The first lemma indicates that certain small $\mathbb Q$-factorialization preserves the structure of QF$^{(\bir)}$ models, possibly after a boundary change:

\begin{lem}\label{lem: qfactorial bqf model}
    Let $(X/Z,B)$ be a pair and $g\colon (Y,B_Y)\dashrightarrow (X,B)$ a QF$^{(\bir)}$ model of $(X/Z,B)$. Then there exists a small $\mathbb Q$-factorialization $h\colon W\rightarrow Y$ that is an isomorphism near the generic point of any lc center of $(Y,B_Y)$ and a pair $(W,B_W)$, such that $g\circ h\colon (W,B_W)\dashrightarrow (X,B)$ is a QF$^{(\bir)}$ model of $(X/Z,B)$, any lc place of $(Y,B_Y)$ is an lc place of $(W,B_W)$ and any lc place of $(W,B_W)$ is an lc place of $(Y,B_Y)$, and $\mathcal{D}(W,B_W)\cong\mathcal{D}(Y,B_Y)$.
\end{lem}
\begin{proof}
    By \cite[Lemma 36]{dFKX17}, there exists a small $\mathbb Q$-factorialization $h\colon W\rightarrow Y$ that is an isomorphism near the generic point of any lc center of $(Y,B_Y)$. Let $K_W+C_W:=h^*(K_Y+B_Y)$. Then any lc place of $(Y,B_Y)$ is an lc place of $(W,C_W)$ and any lc place of $(W,C_W)$ is an lc place of $(Y,C_Y)$, and $\mathcal{D}(W,C_W)\cong\mathcal{D}(Y,B_Y)$.

    We have that $-(K_W+C_W)$ is big$/Z$ and nef$/Z$ and $\Bb_+(-(K_W+C_W)/Z)$ does not contain any lc center of $(W,C_W)$, so there exists $E\geq 0$ on $W$ such that $-(K_W+C_W+E)$ is ample$/Z$ and $\mathcal{D}(W,C_W+E)=\mathcal{D}(W,C_W)$. We may take $B_W:=C_W+E$.
\end{proof}

The second lemma shows that the ambient variety of any pair with non-negative birational strong complete regularity is of Fano type. As we mentioned in the introduction, this indicates that (birational) strong complete regularity is more suitable for Fano type varieties.

\begin{lem}\label{lem: QF is ft}
    Let $(X/Z,B)$ be a pair. Assume that $g\colon (Y,B_Y)\dashrightarrow (X,B)$ is a QF$^{\,\bir}$ model of $(X/Z,B)$. Then $Y$ and $X$ are of Fano type over $Z$.
\end{lem}
\begin{proof}
Let $H_Y:=-(K_Y+B_Y)$. Since $(Y,B_Y)$ is qdlt and $-(K_Y+B_Y)$ is ample$/Z$, by \cite[Lemma 2.3]{XZ25}, there exists a klt pair $(Y,\Delta_Y)$ such that
$$K_Y+B_Y+H_Y\sim_{\mathbb R,Z}K_Y+\Delta_Y$$
and $\Delta_Y$ is big$/Z$. In particular, $Y$ is of Fano type over $Z$. We let $p\colon W\rightarrow Y$ and $q\colon W\rightarrow X$ be a resolution of indeterminacy of $g$. Let $\Delta:=g_*\Delta_Y$. Since $g$ does not extract any divisor, we have that $\Delta$ is big$/Z$. We may write
$$0\sim_{\mathbb R,Z}p^*(K_Y+\Delta_Y)=\colon q^*(K_X+\Delta)+E.$$
Then $E$ is exceptional$/X$. By applying the negativity lemma twice, we have that $E=0$. Thus $(X,\Delta)$ is klt and $K_X+\Delta\sim_{\mathbb R,Z}0$, so $X$ is of Fano type over $Z$.
\end{proof}

Now we are ready to prove Theorem \ref{thm: bqf induces reduced BQF}.

\begin{proof}[Proof of Theorem \ref{thm: bqf induces reduced BQF}]
By Lemma \ref{lem: qfactorial bqf model}, possibly replacing $Y$ with a small $\mathbb Q$-factorialization and replacing $B_Y$ accordingly, we may assume that $Y$ is $\mathbb Q$-factorial. By Lemma \ref{lem: QF is ft}, $Y/Z$ is of Fano type. We let $F:=\Exc(g)-B_Y^{=1}\wedge\Exc(g)$. Since $Y/Z$ is of Fano type, we may run a $F$-MMP$/Z$ (resp. $F$-MMP$/X$) when $(Y/Z,B_Y)$ is a QF$^{\,\bir}$ (resp. QF) model of $(X/Z,B)$, which terminates with a model $X'$ associated with birational map $\phi\colon Y\dashrightarrow X'$ and birational map (resp. morphism) $h\colon X'\rightarrow X$. 

We will show that $\phi$ and $h$, together with $B'$ which we will construct later, satisfy our requirements. Theorem \ref{thm: bqf induces reduced BQF}(2) is satisfied as $X'$ is $\mathbb Q$-factorial.

Let $A_Y:=-(K_Y+B_Y)$. Then $A_Y$ is ample$/Z$, so we may choose $0\leq H_Y\sim_{\mathbb R,Z}A_Y$ such that $(Y,\Delta_Y:=B_Y+H_Y)$ is $\mathbb Q$-factorial qdlt and $\mathcal{D}(Y,\Delta_Y)=\mathcal{D}(Y,B_Y)$. Let $H:=g_*H_Y$. Let $p\colon W\rightarrow Y$ and $q\colon W\rightarrow X$ be a resolution of indeterminacy of $g$, then we have
$$p^*(K_Y+\Delta_Y)=q^*(K_X+B+H).$$
Since $F$ is exceptional$/X$, by \cite[Lemma 3.4(2)]{LX25},
\begin{align*}
   N_{\sigma}(Y/Z,F)&=N_{\sigma}(Y/Z,K_Y+\Delta_Y+F)=\beta_*N_{\sigma}(W/Z,p^*(K_Y+\Delta_Y)+\beta^*F)\\
   &=\beta_*N_{\sigma}(W/Z,q^*(K_X+B+H)+\beta^*F)=\beta_*\beta^*F=F.
\end{align*}
Therefore, by \cite[Lemma 2.25]{LMX24} (resp. \cite[Lemma 3.3]{Bir12}), the divisors contracted by $\phi$ are exactly the irreducible components of $F$, and $\phi_*F=0$. This implies Theorem \ref{thm: bqf induces reduced BQF}(3).

Since $(Y,\Delta_Y)$ is $\mathbb Q$-factorial qdlt and $\mathcal{D}(Y,\Delta_Y)=\mathcal{D}(Y,B_Y)$, by condition (4) of Definition \ref{defn: qdlt model}, $\Supp F$ does not contain any lc center of $(Y,\Delta_Y)$. Thus for any $0<\epsilon\ll 1$, $(Y,\Delta_Y+\epsilon F)$ is $\mathbb Q$-factorial qdlt and $\mathcal{D}(Y,\Delta_Y+\epsilon F)=\mathcal{D}(Y,B_Y)$. Since $K_Y+\Delta_Y\sim_{\mathbb R,Z}0$, $\phi$ is also a sequence of steps of a $(K_Y+\Delta_Y+\epsilon F)$-MMP$/Z$, hence
$$(X',\Delta':=\phi_*\Delta_Y)=(X',\phi_*(\Delta_Y+\epsilon F))$$
is $\mathbb Q$-factorial qdlt. Since $(Y,\Delta_Y)$ and $(X',\Delta')$ are crepant, by \cite[Proposition 11, Corollary 38]{dFKX17}, $\mathcal{D}(Y,\Delta_Y)$ and $\mathcal{D}(X',\Delta')$ are PL-homeomorphic. Moreover, for any prime divisor $D$ over $X$, $D$ is an lc place of $(Y,\Delta_Y)$ if and only if $D$ is an lc place of $(X',\Delta')$.

Since $H_Y$ is ample$/Z$, we have $\Bb_+(H_Y/Z)=\emptyset$. In particular, $\Bb_+(H_Y/Z)$ does not contain any lc center of $(Y,\Delta_Y+\epsilon F)$. By (the qdlt version of) \cite[Lemma 3.10.11(2)]{BCHM10}, we have that $\Bb_+(\phi_*H_Y/Z)$ does not contain any lc center of $(X',\Delta')$. Thus we may write
$$\phi_*H_Y\sim_{\mathbb R,Z}C'+H'$$
where $H'$ is ample$/Z$, $C'\geq 0$, and $\Supp C'$ does not contain any lc center of $(X',\Delta')$. Thus we may find $0<\delta\ll 1$ such that $(X',\Delta'+\delta C')$ is qdlt and $\mathcal{D}(X',\Delta')=\mathcal{D}(X',\Delta'+\delta C')$. Let
$$B':=\phi_*B_Y+(1-\delta)\phi_*H_Y+\delta C',$$
then $\Delta'+\delta C'\geq B'$, so $(X',B')$ is qdlt, and
$$\mathcal{D}(X',B')=\mathcal{D}(B'^{=1})=\mathcal{D}(\Delta'^{=1})=\mathcal{D}(X',\Delta').$$
Therefore, $\mathcal{D}(Y,B_Y)$ and $\mathcal{D}(X',B')$ are PL-homeomorphic, and for any prime divisor $D$ over $X$, $D$ is an lc place of $(Y,B_Y)$ if and only if $D$ is an lc place of $(X',B')$. This implies Theorem \ref{thm: bqf induces reduced BQF}(1).

We are left to check that $h\colon (X',B')\dashrightarrow (X,B)$ is a QF$^{(\bir)}$ model of $(X/Z,B)$. To do so we check Definition \ref{defn: qdlt model}(1)--(4) for  $h\colon (X',B')\dashrightarrow (X,B)$. Definition \ref{defn: qdlt model}(1) is satisfied as
$$B'\geq\phi_*B_Y\geq\phi_*g^{-1}_*B=h^{-1}_*B$$
as $\phi$ does not extract any divisor. Definition \ref{defn: qdlt model}(2) is satisfied by our discussions above. We have that
$$-(K_{X'}+B')\sim_{\mathbb R,Z}\delta H'$$
is ample$/Z$, which implies Definition \ref{defn: qdlt model}(3). Finally, Theorem \ref{thm: bqf induces reduced BQF}(3) implies Definition \ref{defn: qdlt model}(4) as any $h$-exceptional prime divisor $D$ is an irreducible component of $B'^{=1}$. The theorem follows.
\end{proof}

An immediate consequence of Theorem \ref{thm: bqf induces reduced BQF} is the following proposition on the boundedness of complements which preserves the dual complex of QF$^{\,\bir}$ models as a subcomplex.

\begin{prop}
Let $d$ be a positive integer and $\Ii\subset [0,1]$ a DCC set of real numbers. Then there exists a positive integer $N$ depending only on $d$ and $\Ii$ satisfying the following.

Let $(X/Z,B)$ be a pair such that $\dim X=d$ and $B\in\Ii$. Let and $g\colon (Y,B_Y)\dashrightarrow (X,B)$ a QF$^{\,\bir}$ model of $(X/Z,B)$. Then for any point $z\in Z$, possibly shrinking $Z$ to a neighborhood of $z$, there exists an $N$-complement $(X/Z,B^+)$ of $(X/Z,B)$ such that any lc place of $(Y/Z,B_Y)$ is an lc place of $(X/Z,B^+)$, and the image of any lc place of $(X,B^+)$ in $Z$ contains $z$. In particular,
$$\reg(X/Z\ni z,B^+)\geq\reg(Y/Z\ni z,B_Y).$$
 Moreover, if $\overline{\Ii}\subset\mathbb Q$, then we may take $B_Y^+\geq B_Y$. 
\end{prop}
\begin{proof}
By Theorem \ref{thm: bqf induces reduced BQF}, we may assume that $Y$ is $\mathbb Q$-factorial and $(Y/Z,B_Y)$ is a reduced QF$^{\,\bir}$ model of $(X/Z,B)$ associated with $g$. Let $C_Y:=g^{-1}_*B+\Exc(g)$, then $B_Y\geq C_Y$ and $C_Y\in\Ii\cup\{1\}$. By \cite[Theorem 1.5]{HLS24}, over a neighborhood of $z$, there exists a positive integer $N$ depending only on $d$ and $\Ii$, such that $(Y/Z,C_Y)$ has an $N$-complement $(Y/Z,B_Y^+)$. Moreover, by \cite[Theorem 1.2]{FM20}, if  $\overline{\Ii}\subset\mathbb Q$, then we may take $B_Y^+\geq C_Y$. 

Let $B^+:=g_*B_Y^+$, then over a neighborhood of $z$, $(X,B^+)$ and $(Y,B_Y^+)$ are crepant, hence $N(K_X+B^+)\sim 0$ over a neighborhood of $z$ and $(X/Z\ni z,B^+)$ is an $N$-complement of $(X/Z\ni z,B)$. Moreover, if $\overline{\Ii}\subset\mathbb Q$, then we may take $B^+\geq B$. We have that any lc place of $(Y,B_Y)$ is an lc place of $(Y,C_Y)$, hence an lc place of $(Y,B_Y^+)$, and hence an lc place of $(X,B^+)$. The corollary follows.
\end{proof}

\subsection{Qdlt anti-ample models and birational strong complete regularity}

In the study of (birational) strong complete regularity thresholds, and in particular in the proof of the ACC in Section~\ref{sec: acc}, it is useful to work with models that behave well under the minimal model program.
With Theorem \ref{thm: bqf induces reduced BQF} in mind, we introduce \emph{qdlt anti-good log minimal models} and \emph{qdlt anti-ample models}, in the spirit of Shokurov's maximal models \cite[Construction 2]{Sho20}.
Qdlt anti-ample models can also be viewed as a Fano type analogue of Koll\'ar models appearing in K-stability; see Remark \ref{rem: kollar model}.

\begin{defn}[Qdlt anti-good log minimal model]\label{defn: qag model}
Let $(X/Z,B)$ be a pair. A \emph{qdlt anti-good log minimal model} (\emph{QAG model} for short) $\phi\colon (X,B)\dashrightarrow (X',B')$ of $(X/Z,B)$ consists of a birational map$/Z$ $\phi\colon X\dashrightarrow X'$ and a pair $(X'/Z,B')$ satisfying the following.
\begin{enumerate}
    \item $B'=\phi_*(B\vee L)+\Exc(\phi^{-1})$ for some reduced divisor $L$ on $X$.
    \item For any prime divisor $D$ on $X$ that is exceptional$/X'$, we have that 
    $$0<a(D,X',B')<1-\mult_D(B\vee L).$$
    In particular, $a(D,X',B')<a(D,X,B)\leq 1$, and no irreducible component of $L$ is contracted by $\phi$. 
    \item $(X',B')$ is qdlt, $-(K_{X'}+B')$ is semi-ample$/Z$, and $\Bb_+(-(K_{X'}+B')/Z)$ does not contain any lc center of $(X',B')$.
\end{enumerate}
We also say that $(X'/Z,B')$ is a QAG model of $(X/Z,B)$ associated with $(\phi,L)$.
\end{defn}

\begin{defn}[Qdlt anti-ample minimal model]\label{defn: qaa model}
Let $(X/Z,B)$ be a pair. A \emph{qdlt anti-ample model} (\emph{QAA model} for short) $\phi\colon (X,B)\dashrightarrow (X',B')$ of $(X/Z,B)$ consists of a birational map$/Z$ $\phi\colon X\dashrightarrow X'$ and a pair $(X'/Z,B')$ satisfying the following.
\begin{enumerate}
    \item $B'=\phi_*(B\vee L)+\Exc(\phi^{-1})$ for some reduced divisor $L$ on $X$.
    \item For any prime divisor $D$ on $X$ that is exceptional$/X'$, we have that 
    $$0<a(D,X',B')\leq 1-\mult_D(B\vee L),$$
    and if $a(D,X',B')=1$, then $\Center_{X'}D$ is not an lc center of $(X',B')$.
    In particular, $a(D,X',B')\leq a(D,X,B)\leq 1$, and no irreducible component of $L$ is contracted by $\phi$. 
    \item $(X',B')$ is qdlt and $-(K_{X'}+B')$ is ample$/Z$.
\end{enumerate}
We also say that $(X'/Z,B')$ is a QAA model of $(X/Z,B)$ associated with $(\phi,L)$.
\end{defn}

\begin{rem}\label{rem: kollar model}
Consider the case when $\pi\colon X\rightarrow Z$ is the identity morphism, $X$ is affine, and $(X,B)$ is a klt pair. In this case, a QAA model of $(X/X,B)$ associated with $(\phi,0)$ is nothing but a Koll\'ar model in the sense of \cite[Definition 3.7]{XZ25}. Note that when $\pi$ is the identity morphism, any Koll\'ar model is a reduced QF model, and hence a QF$^{\,\bir}$ model. However, the $X\not=Z$ case has significant difference with the $X=Z$. Indeed, a QAA model $(X'/Z,B')$ of $(X/Z,B)$ may not be a QF$^{\,\bir}$ model of $(X/Z,B)$ as the associated birational map $g:=\phi^{-1}\colon X'\dashrightarrow X$ may extract divisors.
\end{rem}

In the rest of this section, we show that there is a correspondence between QAA models, QAG models, and QF$^{\,\bir}$ models, and such correspondence preserves the structure of the dual complexes. In particular, we shall prove the following theorem:
\begin{thm}\label{thm: bscr computed by qag}
    Let $(X/Z,B)$ be a pair. Then 
    \begin{align*}
        \bSReg(X/Z,B)&=\sup\{\reg(X',B')\mid (X'/Z,B')\text{ is a QAG model of }(X/Z,B)\}\\
        &=\sup\{\reg(X',B')\mid (X'/Z,B')\text{ is a QAA model of }(X/Z,B)\}.
    \end{align*}
\end{thm}

We first establish the correspondence between QAA models and QAG models.

\begin{prop}\label{prop: qag to qaa}
 Let $(X/Z,B)$ be an lc pair and $L$ a reduced divisor on $X$.
 \begin{enumerate}
     \item Let $(X'/Z,B')$ be a QAG model of $(X/Z,B)$ associated with $(\phi,L)$ and $f\colon X'\rightarrow Y$ the ample model$/Z$ of $-(K_{X'}+B')$. Then $(Y/Z,B_Y:=f_*B')$ is a QAA model of $(X/Z,B)$ associated with $(f\circ\phi,L)$.
     \item Let $(Y/Z,B_Y)$ be a QAA model of $(X/Z,B)$ associated with $(\psi,L)$. Then there exists a projective birational morphism $f\colon X'\rightarrow Y$ satisfying the following:
     \begin{enumerate}
         \item $X'$ is $\mathbb Q$-factorial.
         \item The divisors contracted by $f$ are exactly the prime divisors $D$ such that $\Center_XD$ is a divisor and $a(D,X,B)=a(D,Y,B_Y)$.
     \end{enumerate}
     In particular, $(X'/Z,B')$ is a $\mathbb Q$-factorial QAG model of $(X/Z,B)$ associated with $(f^{-1}\circ\psi,L)$, where
     $$K_{X'}+B':=f^*(K_Y+B_Y).$$
 \end{enumerate}
\end{prop}
\begin{proof}
    (1) Since $\Bb_+(-(K_{X'}+B')/Z)$ does not contain any lc center of $(X',B')$, $f$ is an isomorphism near the generic point of any lc center of $(X',B')$. (1) follows.

    (2) By \cite[Lemma 36]{dFKX17}, there exists a small $\mathbb Q$-factorialization $\tau\colon Y'\rightarrow Y$ that is an isomorphism near the generic point of any lc center of $(Y,B_Y)$. By \cite[Corollary 1.4.3]{BCHM10}, there exists a birational morphism $f'\colon X'\rightarrow Y'$, such that $X'$ is $\mathbb Q$-factorial, and the divisors contracted by $f'$ are exactly the prime divisors $D$ such that $\Center_XD$ is a divisor and $a(D,X,B)=a(D,Y,B_Y)$. 
    
    Since $\Exc(f')$ of pure codimension $1$ (cf. \cite[Lemma 3.6.2(3)]{BCHM10}) and does not contain any lc place of $(Y,B_Y)$, $f'$ is an isomorphism over the generic point of any lc center of $(Y,B_Y)$ (cf. \cite[Corollary 4.42(3)]{Kol13}). Thus $f:=\tau\circ f'$ is an isomorphism over the generic point of any lc center of $(Y,B_Y)$. Let 
    $$K_{X'}+B':=f^*(K_Y+B_Y),$$ 
    then $(X',B')$ is $\mathbb Q$-factorial qdlt, $-(K_{X'}+B')$ is semi-ample$/Z$, and $\Bb_+(-(K_{X'}+B')/Z)$ does not contain any lc center of $(X',B')$. By our construction, $(X'/Z,B')$ is a $\mathbb Q$-factorial QAG model of $(X/Z,B)$ associated with $(f^{-1}\circ\psi,L)$. (2) follows.
\end{proof}

Next, we show that any QF$^{\,\bir}$ model induces a QAG model preserving the dual complex.

\begin{lem}\label{lem: bqf to qag}
    Let $(X/Z,B)$ be an lc pair and $g\colon (Y,B_Y)\dashrightarrow (X,B)$ a QF$^{\,\bir}$ model of $(X/Z,B)$. Then there exists a $\mathbb Q$-factorial QAG model $(X'/Z',B')$ of $(X/Z,B)$ associated with $(\phi,L:=g_*B_Y^{=1})$ satisfying the following.
    \begin{enumerate}
        \item For any $\phi^{-1}$-exceptional prime divisor $D$, $\Center_YD$ is an irreducible component of $B_Y^{=1}$.
        \item Any lc place of $(Y,B_Y)$ is an lc place of $(X',B')$ and any lc place of $(X',B')$ is an lc place of $(Y,B_Y)$, and $\mathcal{D}(Y,B_Y)$ and $\mathcal{D}(X',B')$ are PL-homeomorphic.
    \end{enumerate}
\end{lem}
\begin{proof}
By Theorem \ref{thm: bqf induces reduced BQF}, we may assume that $(Y/Z,B_Y)$ is a $\mathbb Q$-factorial reduced QF$^{\,\bir}$ model of $(X/Z,B)$. Let $$C_Y:=g^{-1}_*(B\vee L)+\Exc(g),$$ 
then $B_Y\geq C_Y$ and $B_Y^{=1}=C_Y^{=1}$. Let $A_Y:=-(K_Y+B_Y)$, then there exists $0\leq H_Y\sim_{\mathbb R}A_Y$ such that $(Y,\Delta_Y:=B_Y+A_Y)$ is $\mathbb Q$-factorial qdlt and $$\mathcal{D}(Y,\Delta_Y)=\mathcal{D}(Y,B_Y)=\mathcal{D}(Y,C_Y).$$
Let $G_Y:=\Delta_Y-C_Y$. Then $G_Y\geq 0$, and for any $0<\epsilon\ll 1$, $(Y,\Delta_Y+\epsilon G_Y)$ is $\mathbb Q$-factorial qdlt and $\mathcal{D}(Y,\Delta_Y+\epsilon G_Y)=\mathcal{D}(Y,C_Y)$. We have
$$K_Y+\Delta_Y+\epsilon G_Y\sim_{\mathbb R,Z}\epsilon G_Y\quad \text{and}\quad K_Y+C_Y\sim_{\mathbb R,Z}-G_Y.$$
Thus we may run a $(K_Y+\Delta_Y+\epsilon G_Y)$-MMP$/Z$ for some $0<\epsilon\ll 1$, which is also a $-(K_Y+C_Y)$-MMP$/Z$, which terminates with a good minimal model $(X'/Z,\Delta'+\epsilon G')$ of $(Y/Z,\Delta_Y+\epsilon G_Y)$ associated with birational map$/Z$ $\psi\colon Y\dashrightarrow X'$ and $\phi\colon X\dashrightarrow X'$, where $\Delta'$ and $G'$ are the images of $\Delta_Y$ and $G_Y$ on $X'$ respectively. Let $B'$ be the image of $C_Y$ on $X'$. We will show that $\phi\colon (X,B)\dashrightarrow (X',B')$ satisfies our requirements.

Since $\psi$ does not extract any divisor, for any $\phi^{-1}$-exceptional prime divisor $D$, $\Center_YD$ is exceptional$/X$, so $\Center_YD$ is an irreducible component of $\Exc(g)\subset C_Y^{=1}=B_Y^{=1}$. This implies (1). Since $(Y,\Delta_Y+\epsilon G_Y)$ is $\mathbb Q$-factorial qdlt, $(X',\Delta'+\epsilon G')$ is $\mathbb Q$-factorial qdlt. We have
$$\mathcal{D}(Y,B_Y)=\mathcal{D}(Y,\Delta_Y)$$
and
$$\mathcal{D}(X',B')=\mathcal{D}(B'^{=1})=\mathcal{D}(\Delta'^{=1})=\mathcal{D}(X',\Delta').$$
Moreover, $\psi$ is $(K_Y+\Delta_Y)$-trivial, so $(Y,\Delta_Y)$ and $(X',\Delta')$ are crepant. Thus for any prime divisor $D$ over $X$, $D$ is an lc place of $(Y,\Delta_Y)$ if and only if $D$ is an lc place of $(X',\Delta')$. Moreover, \cite[Proposition 11, Corollary 38]{dFKX17}, $\mathcal{D}(X',\Delta')$ and $\mathcal{D}(Y,\Delta_Y)$ are PL-homeomorphic.
This implies (2).

We are left to show that $(X'/Z,B')$ is a $\mathbb Q$-factorial QAG model of $(X/Z,B)$ associated with $(\phi,L)$. $\mathbb Q$-factoriality of $X'$ is clear. We have
$$B'=\psi_*(g^{-1}_*(B\vee L)+\Exc(g))=\phi_*(B\vee L)+\Exc(\phi^{-1})$$
as $\psi$ does not extract any divisor. This implies Definition \ref{defn: qag model}(1). 

Since $\tau$ is a sequence of steps of a $-(K_Y+C_Y)$-MMP, for any prime divisor $D$ on $X$ that is exceptional$/X'$, we have
$$1-\mult_D(B\vee L)=1-\mult_Dg_*C_Y=a(D,Y,C_Y)>a(D,X',B')\geq 0.$$
Moreover, if $a(D,X',B')=0$, then $$a(D,X',\Delta')=a(D,X',\Delta'+\epsilon G')=0,$$
hence 
$$a(D,Y,B_Y+\epsilon G_Y)=a(D,Y,\Delta_Y)=0.$$ 
Thus $a(D,Y,C_Y)=0$, a contradiction. Hence $a(D,X',B')\geq 0$, which implies Definition \ref{defn: qag model}(2). 

By our discussions above, $(X',B')$ is qdlt and $-(K_{X'}+B')$ is semi-ample$/Z$. Since 
$$G_Y=H_Y+(B_Y-C_Y)$$
and $H_Y$ is ample$/Z$, $\Bb_+(G_Y/Z)$ does not contain any lc center of $(Y,B_Y)$. Since $\mathcal{D}(Y,\Delta_Y+\epsilon G_Y)=\mathcal{D}(Y,B_Y)$, $\Bb_+(G_Y/Z)$ does not contain any lc center of $(Y,\Delta_Y+\epsilon G_Y)$. Since $\psi$ is a sequence of steps of a $(K_Y+\Delta_Y+\epsilon G_Y)$-MMP$/Z$, by (the qdlt version of) \cite[Lemma 3.10.11(2)]{BCHM10}, $\Bb_+(G'/Z)$ does not contain any lc center of $(X',\Delta'+\epsilon G')$, hence  $$\Bb_+(G'/Z)=\Bb_+(-(K_{X'}+B')/Z)$$ does not contain any lc center of $(X',\Delta'+\epsilon G')$. Since $0<\epsilon\ll 1$,
$$\mathcal{D}(X',\Delta'+\epsilon G')=\mathcal{D}\left((\Delta'+\epsilon G')^{=1}\right)=\mathcal{D}(\Delta'^{=1})=\mathcal{D}(X',\Delta'),$$ 
this implies Definition \ref{defn: qag model}(3). The lemma follows.
\end{proof}

Finally, we show that any QAG model induces a QF$^{\,\bir}$ model which preserves the dual complex. Together with Lemma \ref{lem: bqf to qag}, we establish a two-side correspondence between  QAG and QF$^{\,\bir}$ models. We first prove the following lemma.

\begin{lem}\label{lem: negativity qag model}
    Let $(X/Z,B)$ be a pair and $(X'/Z,B')$ a QAG model of $(X/Z,B)$ associated with $(\phi,L)$. Let $p\colon W\rightarrow X'$ and $q\colon W\rightarrow X$ be a resolution of indeterminacy of $\phi$. Then
    $$p^*(K_{X'}+B')=q^*(K_{X}+B)+F$$
    for some $F\geq 0$ such that $q_*F=L-B\wedge L+P$ for some $P\geq 0$ such that $\Supp P=\Exc(\phi)$.
\end{lem}
\begin{proof}
We may write $q_*F=P+Q$ where all irreducible components of $P$ are $\phi$-exceptional and no irreducible component of $Q$ is $\phi$-exceptional. By Definition \ref{defn: qag model}(1), $Q=L-B\wedge L\geq 0$. By Definition \ref{defn: qag model}(2), $P\geq 0$ and $\Supp P\supset\Exc(\phi)$. Thus $q_*F\geq 0$. Since
$$-F\sim_{\mathbb R,X}-p^*(K_{X'}+B')$$
is semi-ample$/X$, by the negativity lemma, $F\geq 0$.
\end{proof}

\begin{lem}\label{lem: bqf induce qag}
Let $(X/Z,B)$ be an lc pair, $L$ a reduced divisor on $X$, and $(X'/Z,B')$ a QAG model of $(X/Z,B)$ associated with $(\phi,L)$. Then there exists a $\mathbb Q$-factorial reduced QF$^{\,\bir}$ model $g\colon (Y,B_Y)\dashrightarrow (X,B)$ of $(X/Z,B)$ satisfying the following.
\begin{enumerate}
    \item The associated birational map $\psi\colon Y\dashrightarrow X'$ is a morphism, $\psi$ does not extract any divisor, and $\psi_*B_Y=B'$.
    \item Any $\psi$-exceptional prime divisor is not an lc place of $(X',B')$.
    \item Any lc place of $(Y,B_Y)$ is an lc place of $(X',B')$ and any lc place of $(X',B')$ is an lc place of $(Y,B_Y)$, and $\mathcal{D}(Y,B_Y)$ and $\mathcal{D}(X',B')$ are PL-homeomorphic.
    \item $B_Y\geq g^{-1}_*(B\vee L)+\Exc(g)$.
\end{enumerate}
\end{lem}
\begin{proof}
By \cite[Lemma 36]{dFKX17}, possibly replacing $(X',B')$ with a small $\mathbb Q$-factorialization, we may assume that $X'$ is $\mathbb Q$-factorial. Let $q\colon X'\rightarrow V$ be the ample model$/Z$ of $-(K_{X'}+B')$ and $B_V:=q_*B'$. Since $\Bb(-(K_{X'}+B')/Z)$ does not contain any lc center of $(X',B')$, $q$ is an isomorphism near the generic point of any lc center of $(X',B')$. For any prime divisor $D$ on $X$ that is exceptional$/X'$, we have that
$$0<a(D,X',B')<a(D,X,B)\leq 1.$$
By \cite[Corollary 1.4.3]{BCHM10}, there exists a projective birational morphism $\psi\colon Y\rightarrow X'$ such that $Y$ is $\mathbb Q$-factorial and $\Exc(g)$ is the union of all prime divisors $D$ on $X$ that are exceptional$/X'$. By our construction, the associated birational map $g\colon Y\dashrightarrow X$ does not extract any divisor and the prime divisors contained in $\Exc(g)$ are exactly the prime divisors contained in $\Exc(\phi^{-1})$. Let 
$$K_Y+C_Y:=\psi^*(K_{X'}+B').$$
By \cite[Lemma 3.6.2(3)]{BCHM10}, $\Exc(\psi)$ is of pure codimension $1$ and does not contain any lc place of $(X',B')$ Thus $\psi$ is an isomorphism near the generic point of any lc center of $(Y,C_Y)$ as any irreducible component of $\overline{p^{-1}(\eta)}$ is an lc center of $(Y,C_Y)$ for any generic point $\eta$ of an lc center of $(X',B')$ (cf. \cite[Corollary 4.42(3)]{Kol13}). In particular, $(Y,C_Y)$ is $\mathbb Q$-factorial qdlt and $q\circ\psi$ is an isomorphism near the generic point of any lc center of $(Y,C_Y)$. Thus we may find $R\geq 0$ on $Y$ that is anti-ample$/V$ and $\Supp R$ does not contain any lc center of $(Y,C_Y)$. Since $-(K_V+B_V)$ is ample$/Z$, possibly rescaling $R$, we may assume that
$$(q\circ p)^*(-(K_V+B_V))-R=-(K_Y+C_Y+R)$$
is ample$/Z$, $(Y,B_Y:=C_Y+R)$ is $\mathbb Q$-factorial qdlt, and $$\mathcal{D}(Y,B_Y)=\mathcal{D}(Y,C_Y)=\mathcal{D}\left(C_Y^{=1}\right).$$
We show that $g\colon (Y,B_Y)\dashrightarrow (X,B)$ is a $\mathbb Q$-factorial reduced QF$^{\,\bir}$ model of $(X/Z,B)$. For any prime divisor $D$ on $Y$ that is exceptional$/X$, $\psi_*D$ is a prime divisor and is exceptional$/X$, hence
$$\mult_DB_Y=\mult_{\psi_*D}B'=1.$$
For any prime divisor $D$ on $Y$ that is not exceptional$/X$, by Lemma \ref{lem: negativity qag model},
$$\mult_DB_Y\geq\mult_DC_Y=1-a(D,Y,C_Y)=1-a(D,X',B')\geq 1-a(D,X,B)=\mult_{g_*D}B.$$
Thus $g\colon (Y,B_Y)\dashrightarrow (X,B)$ is a $\mathbb Q$-factorial reduced QF$^{\,\bir}$ model of $(X/Z,B)$.

Finally, we show (1-4). (1)(2) immediately follow from our construction. Since $(Y,C_Y)$ and $(X',B')$ are crepant, any lc place of $(Y,C_Y)$ is an lc place of $(X',B')$ and any lc place of $(X',B')$ is an lc place of $(Y,C_Y)$. Moreover, by \cite[Proposition 11, Corollary 38]{dFKX17}, $\mathcal{D}(Y,C_Y)$ and $\mathcal{D}(X',B')$ are PL-homeomorphic. Since $\mathcal{D}(Y,B_Y)=\mathcal{D}(Y,C_Y)$, we get (3). Since $\phi$ does not contract any irreducible component of $L$ and
$$\psi_*B_Y=B'\geq \phi_*L,$$
we have $g_*B_Y\geq L$. Since $g\colon (Y,B_Y)\dashrightarrow (X,B)$ is a $\mathbb Q$-factorial reduced QF$^{\,\bir}$ model of $(X/Z,B)$, we obtain (4). The lemma follows.
\end{proof}

\begin{proof}[Proof of Theorem \ref{thm: bscr computed by qag}]
    It is an immediate consequence of Proposition \ref{prop: qag to qaa} and Lemmas \ref{lem: bqf to qag} and \ref{lem: bqf induce qag}. 
\end{proof}

\medskip
The results of this section will be used repeatedly in the proofs of Theorem~\ref{thm: 1-complementary} and Theorem~\ref{thm: acc scrt} in the following sections.

\section{Complements}\label{sec: complement}

The goal of this section is to show that pairs with maximal birational strong complete regularity are $1$-complementary. More precisely, we prove the following theorem in this section:

\begin{thm}\label{thm: 1complementary scr max}
    Let $d$ be a positive integer. Assume that $(X/Z\ni z,B)$ is an lc pair of dimension $d$ such that $\bSReg(X/Z\ni z,B)=d-1-\dim z$.

    Then there exists a $1$-complement $(X/Z\ni z,B^+)$ of $(X/Z\ni z,B)$ such that
    $$\reg(X/Z\ni z,B^+)=d-1-\dim z.$$ 
    In particular, $(X/Z\ni z,B)$ is $1$-complementary. 
\end{thm}

Theorem \ref{thm: 1complementary scr max} shows the essential difference between (birational) strong complete regularity and complete regularity: pairs with maximal complete regularity are only guaranteed to be $1$-complementary or $2$-complementary \cite{Asc+23,Mor24a}, whereas pairs with maximal birational strong complete regularity are always $1$-complementary.

To prove Theorem \ref{thm: 1complementary scr max}, we start with the following proposition.

\begin{prop}\label{prop: lift 1 complement}
Assume that:
\begin{enumerate}
    \item $(X/Z\ni z,B)$ is an lc pair wit contraction $\pi\colon X\rightarrow Z$.
    \item $-(K_X+B)$ is nef$/Z$. 
    \item $(X,C)$ is plt.
    \item $\{C\}\geq\{B\}$.
    \item $-(K_X+C)$ is big$/Z$ and nef$/Z$.
    \item The non-zero coefficients of $B$ are $\geq\frac{1}{2}$.
    \item $S=C^{=1}\subset B^{=1}$.
    \item $z\in\pi(S)$.
    \item $K_S+B_S:=(K_X+B)|_S$, and
    \item $(S/Z,B_S^+)$ is a $1$-complement of $(S/Z,B_S)$.
\end{enumerate}
Then there exists a $1$-complement $(X/Z\ni z,B^+)$ of $(X/Z\ni z,B)$ such that
$$(K_X+B^+)|_S=K_S+B_S^+$$
over a neighborhood of $z$.
\end{prop}
\begin{proof}
    The proof of the proposition is similar to \cite[Proposition 6.2]{PS01} (but we cannot directly apply it since $(X,B)$ may not be plt) and \cite[Proposition 6.7]{Bir19} (but we have a much weaker restriction on the coefficients of $B$). For the reader's convenience, we provide a full proof here.

Let $g\colon W\to X$ be a log resolution of $(X,B+C)$,
	$$-N_W:=K_W+B_W:=g^{*}(K_X+B), \quad \quad K_W+C_W:=g^*(K_X+C),$$
	and $S_W:=g^{-1}_*S$. Let $T_W:=B_W^{=1}$ and $\Delta_W:=B_W-T_W$. Possibly replacing $C$ with $\epsilon C+(1-\epsilon)B$ for some $0<\epsilon\ll 1$, we may assume that $||B_W-C_W||$ is sufficiently small. In particular, we may assume that:
\begin{itemize}
\item $\Supp C_W\supset\Supp B_W$ and $\Supp C\supset\Supp B$.
\item For any component $D$ of $\Delta_W$, 
$$-1<\mult_D(C_W-\Delta_W+\{2\Delta_W\})<1.$$
\item For any irreducible component $D$ of $T_W$, $\mult_DC_W>0$.
\item For any irreducible component $D$ of $C_W$ that is not an irreducible component of $B_W$, $$-1<\mult_DC_W<1.$$
\end{itemize}
We let $P_W$ be the unique Weil divisor on $W$ satisfying the following:
\begin{itemize}
    \item $\mult_{S_W}P_W:=0$.
    \item For any prime divisor $D$ on $W$ such that $D\not=S_W$, we have
    $$\mult_DP_W:=-\mult_{D_W}\left\lfloor C_W+\Delta_W-\left\lfloor 2\Delta_W\right\rfloor\right\rfloor.$$
\end{itemize}
\begin{claim}\label{claim: exceptional PW}
    $P_W$ is a reduced divisor and $P_W$ is exceptional$/X$.
\end{claim}
\begin{proof}
    First we show that $P_W$ is a reduced divisor. We only need to show that $\mult_DP_W\in\{0,1\}$ for any prime divisor $D$ on $W$. We may assume that $D\not=S_W$. Then
    $$\mult_DP_W=-\mult_{D}\left\lfloor C_W-\Delta_W+\left\{2\Delta_W\right\}\right\rfloor.$$
    If $D$ is an irreducible component of $T_W$, then $D$ is not an irreducible component of $\Delta_W$ but is an irreducible component of $C_W^{\geq 0}$. Since $\left\lfloor C_W^{\geq 0}\right\rfloor=S_W$, $\mult_DC_W\in (0,1)$, hence
    $$\mult_DP_W=-\mult_D\left\lfloor C_W\right\rfloor=0.$$
    If $D$ is an irreducible component of $\Delta_W$, then since 
    $$-1<\mult_D(C_W-\Delta_W+\{2\Delta_W\})<1,$$ 
    we have
    $\mult_DP_W\in\{0,1\}$. If $D$ is an irreducible component of $C_W$ but is not an irreducible component of $\Delta_W$, then since $D\not=S_W$, $-1<\mult_DC_W<1$, so  $\mult_DP_W\in\{0,1\}$. 

    Next we show that $P_W$ is exceptional$/X$. For any irreducible component $D$ of $P_W$ that is not exceptional$/X$, we have
    $$1=\mult_DP_W=-\mult_{g_*D}\left\lfloor C-B+\{2B\}\right\rfloor.$$
    Thus $g_*D$ is an irreducible component of $C$ or $B$ and $g_*D\not=S$. If $g_*D$ is not an irreducible component of $B$, then
    $$1=-\mult_{g_*D}\lfloor C\rfloor=0,$$
    a contradiction. If $g_*D$ is an irreducible component of $\lfloor B\rfloor$, then
    $$1=-\mult_{g_*D}\lfloor C-B\rfloor=1-\mult_{g_*D}C,$$
    so $\mult_{g_*D}C=0$, which is not possible as $\Supp C\supset\Supp B$. If $g_*D$ is an irreducible component of $\{B\}$, then since $\{C\}\geq\{B\}$,
    $$1=-\mult_{g_*D}\lfloor (C-B)+\{2B\}\rfloor\leq 0,$$
    a contradiction. Therefore, $P_W$ is exceptional$/X$.
\end{proof}
\noindent\textit{Proof of Proposition \ref{prop: lift 1 complement} continued.} Since $S$ is normal, the induced birational morphism $g|_{S_W}\colon S_W\to S$ is a birational contraction. Therefore,
	$$-N_W|_{S_W}=(K_W+B_W)|_{S_W}=g|_{S_W}^{*}(K_S+B_S).$$
Let $R_S:=B_S^+-B_S$. Since $B\in\left[\frac{1}{2},1\right]$, by adjunction (cf. \cite[16.7 Corollary]{Kol+92}), $B_S\in\left[\frac{1}{2},1\right]$. Since $(S/Z,B_S^+)$ is a $1$-complement of $(S/Z,B_S)$, $R_S\geq 0$. Let $R_{S_W}:=g|_{S_W}^*R_S$, then
$R_{S_W}\geq 0$, and 
$$N_W|_{S_W}\sim g|_{S_W}^{*}(K_S+B_S^+)-g|_{S_W}^{*}(K_S+B_S)=R_{S_W}.$$
Since $(W,\Supp B_W\cup\Supp C_W)$ is log smooth and $S_W$ is not an irreducible component of $\Delta$ nor an irreducible component of $P_W$, we may define $\Delta_{S_W}:=\Delta|_{S_W}$, $P_{S_W}:=P_W|_{S_W}$, 
$$L_W:=-K_W-T_W-\left\lfloor 2\Delta_W\right\rfloor,\quad \text{and}\quad G_{S_W}:=R_{S_W}+\Delta_{S_W}-\left\lfloor 2\Delta_{S_W}\right\rfloor+P_{S_W}.$$
Then we have
$$(L_W+P_W)|_{S_W}\sim G_{S_W}.$$
In particular, $G_{S_W}$ is a Weil divisor. Thus for any irreducible component $Q$ of $G_{S_W}$, we have
$$\mult_{Q}G_{S_W}\geq\mult_{Q}(\Delta_{S_W}-\lfloor 2\Delta_{S_W}\rfloor)=\mult_{Q}(-\Delta_{S_W}+\{2\Delta_{S_W}\})\geq-\mult_{Q}\Delta_{S_W}>-1,$$
so $G_{S_W}\geq 0$. We define
$$\Lambda_W:=C_W+\Delta_W-\left\lfloor2\Delta_W\right\rfloor+P_W.$$
By our construction of $P_W$,
$$\Lambda_W=S_W+\sum_{D\not=S_W}\mult_D\{C_W+\Delta_W-\lfloor2\Delta_W\rfloor\}$$
where the sum runs through all prime divisors $D$ on $W$ that are irreducible components of $C_W$ or $B_W$. In particular, $(W,\Lambda_W)$ is log smooth, plt, and $\left\lfloor\Lambda_W\right\rfloor=S_W$. We have
\begin{align*}
	L_W+P_W-S_W&=N_W+\Delta_W-\lfloor 2\Delta_W\rfloor+P_W-S_W\\
	&=N_W+\Lambda_W-C_W-S_W\\
	&=K_W+(\Lambda_W-S_W)+(-(K_W+C_W)+N_W).
\end{align*}
Let $h:=\pi\circ g$. Since $(W,\Lambda_W-S_W)$ is klt and $-(K_W+C_W)+N_W$ is big$/Z$ and nef$/Z$, by a version of the relative Kawamata-Viehweg vanishing theorem (cf. \cite[Theorem 1-2-5]{KMM87}, \cite[Theorem 3.2.9]{Fuj17}), 
$$R^1h_*\mathcal{O}_W(L_W+P_W-S_W)=0$$
From the exact sequence
	$$0\to\mathcal{O}_W(L_W+P_W-S_W)\to \mathcal{O}_W(L_W+P_W)\to \mathcal{O}_{S_W}((L_W+P_W)|_{S_W})\to 0,$$
	we deduce that
	$$R^ih_*\mathcal{O}_W(L_W+P_W)\rightarrow R^ih_*\mathcal{O}_{S_W}((L_W+P_W)|_{S_W})$$
	is surjective, so there exists $G_W\geq 0$ on $W$ such that $G_W|_{S_W}=G_{S_W}$ and $G_W\sim_Z L_W+P_W$.

    Let $G:=g_*G_W$. By Claim \ref{claim: exceptional PW}, $P_W$ is exceptional$/X$. Since the coefficients of $B$ are $\geq\frac{1}{2}$, we have 
$$G\sim_Z g_*L_W=-K_X-\lfloor B\rfloor-\lfloor\{2B\}\rfloor=-K_X-\Supp B.$$
Thus $K_X+\Supp B+G\sim_Z 0$. Let 
    $$R_W:=G_W-P_W-\Delta_W+\lfloor 2\Delta_W\rfloor.$$
    Then
    $$R_W\sim_Z L_W-\Delta_W+\lfloor 2\Delta_W\rfloor=N_W\sim_{\mathbb R,X}0.$$
 Moreover,
    $R_W|_{S_W}=R_{S_W}$ by the definition of $G_{S_W}$. Let $R:=g_*R_W$. By Claim \ref{claim: exceptional PW},
    $$R=G-\{B\}+\lfloor\{2B\}\rfloor=G+\Supp\{B\}-\{B\}.$$
    Since $R_{S_W}=g|_{S_W}^*R_S$, we have
    $R_S=R|_S$, hence
    $$K_S+B_S^+=(K_X+B+R)|_{S}.$$
    By inversion of adjunction, 
    $$(X,B^+:=B+R=\Supp B+G)$$
    is lc near $S$. 
    
    Suppose that $(X,B^{+})$ is not lc over a neighborhood of $z$. then there exists a real number $0<\delta\ll 1$, such that over a neighborhood of $z$ $(X,(1-\delta)B^{+}+\delta C)$ is not lc and is lc near $S$. In particular, $\Nklt(X,aB^{+}+(1-a)C)\cap\pi^{-1}(z)\not=\emptyset$. But
$$-(K_X+(1-\delta)B^{+}+\delta C)=-(1-\delta)(K_X+B^{+})-\delta(K_X+C)\sim_{\mathbb R}-\delta(K_X+C)$$
is big$/Z$ and nef$/Z$. This contradicts the Shokurov-Koll\'ar connectedness principle.
\end{proof}

\begin{lem}\label{lem: qdlt ample coefficient larger than 1/2 case}
Assume that Theorem \ref{thm: 1complementary scr max} holds in dimension $\leq d-1$. 

Let $(X/Z\ni z,B)$ be a pair of dimension $d$ such that $(X,B)$ is qdlt, $-(K_X+B)$ is ample$/Z$, the non-zero coefficients of $B$ are $\geq\frac{1}{2}$, and
$$\reg(X/Z\ni z,B)=\dim X-1-\dim z.$$
Then $(X/Z\ni z,B)$ has a monotonic $1$-complement.
\end{lem}
\begin{proof}
By \cite[Lemma 36]{dFKX17}, there exists a small $\mathbb Q$-factorialization $h\colon X'\rightarrow X$ that is an isomorphism near the generic point of any lc center of $(X,B)$. Write $K_{X'}+B':=h^*(K_X+B)$, then we may write
$$-(K_{X'}+B')=A_k+\frac{1}{k}E$$
for any positive integer $k$, where $A_k$ are ample$/Z$ $\mathbb R$-divisors, $E\geq 0$, and $\Supp E$ does not contain any lc center of $(X',B')$. Pick $n\gg 0$, then $\left(X',B'+\frac{1}{n}E'\right)$ is qdlt and $\mathcal{D}\left(X',B'+\frac{1}{n}E\right)=\mathcal{D}(X',B')$. Possibly shrinking $Z$ to a neighborhood of $z$, we may assume that $z$ is contained in the image of any irreducible component of $\left\lfloor B'\right\rfloor$ on $Z$. Let $S'$ be an irreducible component of $\left\lfloor B'\right\rfloor$. Then we may pick $0<\epsilon\ll\frac{1}{n}$ such that
$$-\left(K_{X'}+B'-\epsilon\left(\left\lfloor B'\right\rfloor-S'\right)+\frac{1}{n}E\right)$$
is ample$/Z$. Let 
$$C':=B'-\epsilon\left(\left\lfloor B'\right\rfloor-S'\right)+\frac{1}{n}E\quad \text{and}\quad S:=h_*S',$$
and let $\pi':=X'\rightarrow Z$ be the associated contraction. Then by our construction:
\begin{itemize}
    \item $(X'/Z,B')$ is an lc pair with associated contraction $\pi'\colon X'\rightarrow Z$.
    \item $-(K_{X'}+B')$ is nef$/Z$.
    \item $(X',C')$ is plt.
    \item $\{C'\}=\left\{B'+\frac{1}{n}E\right\}\geq\{B'\}$.
    \item $-(K_{X'}+C')$ is ample$/Z$, hence big$/Z$ and nef$/Z$.
    \item The non-zero coefficients of $B$ are $\geq\frac{1}{2}$, hence the non-zero coefficients of $B'$ are $\geq\frac{1}{2}$.
    \item $S'=\left\lfloor C'\right\rfloor\subset\left\lfloor B'\right\rfloor$.
    \item $z\in\pi'(S')$.
\end{itemize}
Let
$$K_{S'}+B_{S'}:=(K_{X'}+B')|_{S'}\quad\text{and}\quad K_S+B_S:=(K_X+B)|_S,$$
then we have
$$K_{S'}+B_{S'}=h|_{S'}^*(K_X+B_S).$$
Since $-(K_X+B)$ is ample$/Z$ and $(X,B)$ is qdlt, by \cite[Lemma 3.11]{XZ25}, $S$ intersects any irreducible component of $B^{=1}$ over a neighborhood of $z$. Therefore, $(S,B_S)$ is qdlt, $-(K_S+B_S)$ is ample$/Z$, and 
$$\dim\mathcal{D}(S,B_S)=\dim\mathcal{D}\left(B_S^{=1}\right)=\dim S-1-\dim z$$
over a neighborhood of $z$. Thus
$$\reg(S/Z\ni z,B_S)=\dim S-1-\dim z.$$
Recall that we have assumed that Theorem \ref{thm: 1complementary scr max} holds in dimension $\leq d-1$. Thus possibly shrinking $S$ to a neighborhood of $z$, $(S/Z,B_S)$ has a $1$-complement $(S/Z,B_S^+)$. By Proposition \ref{prop: lift 1 complement}, $(X'/Z\ni z,B')$ has a $1$-complement $(X'/Z\ni z,B'^+)$. Let $B^+:=h_*B'^+$, then $(X/Z\ni z,B^+)$ is a $1$-complement of $(X/Z\ni z,B)$. Since the non-zero coefficients of $B$ are $\geq\frac{1}{2}$, $(X/Z\ni z,B^+)$ is a monotonic $1$-complement of $(X/Z\ni z,B)$. The lemma follows.
\end{proof}

\begin{proof}[Proof of Theorem \ref{thm: 1complementary scr max}]
Possibly replacing $X$ with a small $\mathbb Q$-factorialization, we may assume that $X$ is $\mathbb Q$-factorial. Let $C:=B^{\geq\frac{1}{2}}$. By Proposition \ref{prop: compare creg},
$$d-1-\dim z\geq \bSReg(X/Z\ni z,C)\geq\bSReg(X/Z\ni z,B)=d-1-\dim z,$$
and any $1$-complement of $(X/Z\ni z,C)$ is a $1$-complement of $(X/Z\ni z,B)$. Thus possibly replacing $B$ with $C$, we may assume that all non-zero coefficients of $B$ are $\geq\frac{1}{2}$. By Theorem \ref{thm: bscr computed by qag}, possibly shrinking $Z$ to a neighborhood of $z$, there exists a QAA model $(X'/Z,B')$ of $(X/Z,B)$ associated with $(\phi,L)$ for some reduced divisor $L$ such that
$$\reg(X'/Z\ni z,B')=\reg(X'/Z,B')=d-1-\dim z.$$
By Lemma \ref{lem: qdlt ample coefficient larger than 1/2 case} and induction on dimension, possibly shrinking $Z$ to a neighborhood of $z$, there exists a monotonic $1$-complement $(X'/Z,B'^+)$ of $(X'/Z,B')$. Let $G':=B'^+-B'$ and let $p\colon W\rightarrow X'$ and $q\colon W\rightarrow X$ be a resolution of indeterminacy of $\phi$. By Lemma \ref{lem: negativity qag model},
$$p^*(K_{X'}+B')=q^*(K_X+B)+F$$
for some $F\geq 0$. Let $G:=q_*(F+p^*G')$, then $G\geq 0$ and $(X,B^+:=B+G)$ and $(X',B'+G')$ are crepant. Thus $(X/Z,B^+)$ is a $1$-complement of $(X/Z,B)$ and
$$\reg(X/Z\ni z,B^+)=\reg(X'/Z\ni z,B'^+)=d-1-\dim z.$$
The theorem follows.
\end{proof}

\begin{proof}[Proof of Theorem \ref{thm: 1-complementary}]
    It follows from Theorem \ref{thm: 1complementary scr max} and Proposition \ref{prop: compare creg}.
\end{proof}

\section{ACC for strong complete regularity threshold}\label{sec: acc}

The goal of this section is to show that the (birational) strong complete regularity thresholds (Definition \ref{defn: scrt}) satisfy the ACC condition. We begin with a basic observation showing that these thresholds are always attained:

\begin{prop}\label{prop: scrt is computed}
Let $d>n\geq 0$ be two integers. Let $(X/Z,B)$ be a pair of dimension $d$ and $D\geq 0$ an $\mathbb R$-Cartier $\mathbb R$-divisor on $X$. Assume that
 $$t:=\SRegt^{(\bir)}_n(X/Z,B;D)\in (0,+\infty).$$
 Then either $t=\lct(X,B;D)$, or $\SReg^{(\bir)}(X/Z,B+tD)<n$.
\end{prop}
\begin{proof}
We may assume that $t<\lct(X,B;D)$ but $\SRegt^{(\bir)}(X/Z,B;D)\geq n$. By Lemma \ref{lem: qfactorial bqf model}, there exists a QF$^{(\bir)}$ model $g\colon (Y,B_Y)\dashrightarrow (X,B+tD)$ of $(X/Z,B+tD)$ such that $Y$ is $\mathbb Q$-factorial and $\reg(Y,B_Y)\geq n$. Let $D_Y$ be the strict transform of $D$ on $Y$. Since $t<\lct(X,B;D)$, $\Supp D\subset\Supp\{B+tD\}$, hence 
$$\Supp D_Y\subset\Supp\{B_Y\}=\Supp\{B_Y+E\}.$$
Thus we may pick $0<\epsilon\ll 1$ such that $(Y,\Delta_Y:=B_Y+\epsilon D_Y)$ is qdlt and $-(K_Y+\Delta_Y)$ is ample$/Z$. We have that 
$$g\colon(Y,\Delta_Y)\dashrightarrow (X,B+(t+\epsilon)D)$$
is a QF$^{(\bir)}$ model of $(X/Z,B+(t+\epsilon)D)$, so
$$\SReg^{(\bir)}(X/Z,B+(t+\epsilon)D)\geq\reg(Y,\Delta_Y)\geq\reg(Y,B_Y)=n.$$
Thus $\SRegt^{(\bir)}_n(X/Z,B;D)\geq t+\epsilon$, a contradiction.
\end{proof}

The goal of this section is to prove the following theorem.

\begin{proof}[Proof of Theorem \ref{thm: acc scrt}]
Suppose that theorem does not hold. Then for any positive integer $i$, there exist a pair $(X_i/Z_i,B_i)$ of dimension $d$, an integer $n_i$ such that $0\leq n_i<d$, and an $\mathbb R$-Cartier $\mathbb R$-divisor $D_i$ on $X_i$, such that $B_i,D_i\in\Ii$, but
$$t_i:=\SRegt^{(\bir)}_n(X_i/Z_i,B_i;D_i)$$
is strictly increasing. Possibly passing to a subsequence, we may assume that $t_i\in (0,+\infty)$ for any $i$ and $n:=n_i$ is a constant. By the ACC for lc thresholds, \cite[Theorem 1.1]{HMX14}, possibly passing to a subsequence, we may assume that $t_i<\lct(X_i,B_i;D_i)$ for any $i$. Let $t:=\lim_{i\rightarrow+\infty}t_i$. Then we may pick a sequence of real numbers $\{s_i\}_{i=1}^{+\infty}$, such that $s_i$ is strictly increasing, $s_i<t_i$ for each $i$, and $\lim_{i\rightarrow+\infty}s_i=t$. Then
$$\SReg^{(\bir)}(X_i/Z_i,B_i+s_iD_i)\geq n$$
for any $i$. By Theorem \ref{thm: bqf induces reduced BQF}, for each $i$, there exists a reduced divisor $L_i$ on $X_i$ and a $\mathbb Q$-factorial reduced QF$^{(\bir)}$ model 
$$g_i\colon (Y_i,\Delta_{Y_i}:=B_{Y_i}+s_iD_{Y_i}+R_i)\dashrightarrow (X_i,B_i+s_iD_i)$$ 
of $(X_i/Z_i,B_i+s_iD_i)$ such that $Y_i$ is $\mathbb Q$-factorial and $\reg(Y_i,B_{Y_i})\geq n$, where $$B_{Y_i}:=g_{i,*}^{-1}(B_i\vee L_i)+E_i$$ for some reduced divisor $L_i$ on $X_i$ and $g_i$-exceptional reduced divisor $E_i$, $D_{Y_i}:=g_{i,*}^{-1}D_i$, and $R_i\geq 0$. We have that $(Y_i/Z_i,\Delta_{Y_i})$ is $\mathbb R$-complementary, hence $(Y_i/Z_i,B_{Y_i}+s_iD_{Y_i})$ is $\mathbb R$-complementary. By Lemma \ref{lem: QF is ft}, $Y_i/Z_i$ is of Fano type. By the ACC for $\mathbb R$-complementary thresholds \cite[Theorem 5.12]{HLS24}, possibly passing to a subsequence, we have that $(Y_i/Z_i,B_{Y_i}+tD_{Y_i})$ is $\mathbb R$-complementary. Let $$\left(Y_i/Z_i,B_{Y_i}^+:=B_{Y_i}+tD_{Y_i}+G_{i}\right)$$
be an $\mathbb R$-complement of $(Y_i/Z_i,B_{Y_i}+tD_{Y_i})$. Pick $0<u_i\ll t-t_i$. Consider
$$\left(Y_i/Z_i,C_{i}:=B_{Y_i}+((1-u_i)t+u_is_i)D_{Y_i}+(1-u_i)G_{i}+u_iR_i\right).$$
Then:
\begin{enumerate}
    \item $g_i$ does not extract any divisor and $C_i\geq g^{-1}_{i,*}(B_{i}+t_iD_i)$.
    \item Since 
    $$K_{Y_i}+C_i\sim_{\mathbb R,Z_i}(1-u_i)\left(K_{Y_i}+B_{Y_i}^+\right)+u_i(K_{Y_i}+\Delta_{Y_i})\sim_{\mathbb R,Z_i}u_i(K_{Y_i}+\Delta_{Y_i})$$
    is ample$/Z_i$. Moreover, any lc place of $(Y_i,C_i)$ is an lc place of $(Y_i,\Delta_{Y_i})$, hence $(Y_i,C_i)$ is qdlt.
    \item Any $g_i$-exceptional prime divisor $D$ such that $\mult_{D}C_i=0$ satisfies that $\mult_D\Delta_{Y_i}=0$, so $D$ does not pass through any lc center of $(Y_i,\Delta_{Y_i})$, and so $D$ does not pass through any lc center of $(Y_i,C_i)$.
\end{enumerate}
Therefore, 
$$g_i\colon (Y_i,C_i)\dashrightarrow (X_i,B_{i}+t_iD_i)$$
is a QF$^{(\bir)}$ model of $(X_i,B_i+t_iD_i)$. Thus
$$\SReg^{(\bir)}(X_i/Z_i,B_i+t_iD_i)\geq\reg(Y_i,C_i)=\reg(Y_i,\Delta_{Y_i})\geq n.$$
This contradicts Proposition \ref{prop: scrt is computed}.
\end{proof}

\end{document}